\RequirePackage{booktabs}

\documentclass[sn-mathphys-num]{sn-jnl}% Math and Physical 

\usepackage{amsmath,amsfonts,amsthm,amssymb}	
\usepackage{todonotes}
\usepackage{bbm}
\usepackage{comment}
\usepackage{siunitx}
\usepackage{subfigure}
\usepackage{algpseudocode}
\usepackage{algorithm}
\usepackage{rotating}
\usepackage{adjustbox}
\usepackage{graphicx}
\usepackage{bigstrut}
\usepackage{multicol}
\usepackage{multirow}
\usepackage[english]{babel}
\usepackage[T1]{fontenc}
\usepackage{caption}
\usepackage{numprint}
\usepackage{footnote}
\makesavenoteenv{tabular} 
\makesavenoteenv{table}

\newcommand{\ds}{\displaystyle}

\newtheorem{proposition}{Proposition}
\newtheorem{lemma}[proposition]{Lemma}

\newtheorem{assumption}{Assumption}

\newtheorem{remark}{Remark}
\newcommand{\R}{{\mathbb R}}
\newcommand{\N}{{\mathbb N}}

\renewcommand{\Re}{\mathbb{R}}

\newcommand{\commento}[1]{}
\newcommand{\CMA}{{\texttt{CMA}}}
\newcommand{\NMCMA}{{\texttt{NMCMA}}}
\usepackage{curve2e}
\definecolor{gray}{gray}{0.4}

\begin{document}

\title[Article Title]{Convergence  of ease-controlled Random Reshuffling gradient Algorithms under Lipschitz smoothness}

\author[1]{\fnm{Ruggiero} \sur{Seccia}}\email{ruggiero.seccia@outlook.com}

\author[1]{\fnm{Corrado} \sur{Coppola}}\email{corrado.coppola@uniroma1.it}
\author[1]{\fnm{Giampaolo} \sur{Liuzzi}}\email{giampaolo.liuzzi@uniroma1.it}

\author[1]{\fnm{Laura} \sur{Palagi}}\email{laura.palagi@uniroma1.it}

\affil*[1]{\orgdiv{Department of Computer, Control and Management Engineering A. Ruberti}, \orgname{Sapienza, University of Rome}, \orgaddress{\street{Via Ariosto 25}, \city{Rome}, \postcode{00185}, \state{RM}, \country{Italy}}}


\abstract{In this work, we consider minimizing the average of a very large number of smooth and possibly non-convex functions and we focus on two widely used minibatch frameworks to tackle this optimization problem: Incremental Gradient (IG) and Random Reshuffling (RR). 
% \remove[RS]{Convergence properties of these schemes have been proved under different assumptions, usually quite strong.} 
We define ease-controlled modifications of the IG/RR schemes, which require a light additional computational effort  {but} can be proved to converge under {weak} and standard assumptions.% \add[RS]{compared to the assumptions usually introduced to show convergence of these schemes}. 
 In particular, we define two algorithmic schemes in which the IG/RR iteration is controlled by using a watchdog rule and a derivative-free linesearch that activates only sporadically to guarantee convergence. The two schemes differ in the watchdog and the linesearch, which are performed using either a monotonic or a non-monotonic rule. The two schemes also allow controlling the updating of the stepsize used in the main IG/RR iteration, avoiding the use of pre-set rules that may drive the {stepsize} to zero too fast, %\change[RS]{thus overcoming another challenging aspect in implementing online methods, which is the updating rule}
 {reducing the effort in designing effective updating rules}  of the stepsize. We prove convergence under the mild assumption of Lipschitz continuity of the gradients of the component functions and perform extensive computational analysis using different deep neural architectures and a benchmark of varying-size datasets. We compare our implementation with both a full batch gradient method (i.e. L-BFGS) and an implementation of IG/RR methods, proving that our algorithms require a similar computational effort compared to the other online algorithms and that the control on the learning rate may allow a faster decrease of the objective function.}




\maketitle
%\newpage

\section{Introduction}\label{sec1}

In this paper, we consider the finite-sum optimization problem

\begin{equation}\label{eq: problem}
    \underset{w\in \R^n}{\min }\qquad f(w)=\sum_{p=1}^Pf_p(w)
\end{equation}
where $f_p:\R^n\to \R$, $p=1,\dots,P$ are  continuously differentiable and possibly non-convex functions.

Problem \eqref{eq: problem} is a well-known optimization problem that plays a key role in many applications and has consequently attracted researchers from many different fields. Some applications can be found in e.g. methods for solving large-scale feasibility problems by minimizing the constraints violation \cite{mayne1981solving,zhang2010nonmonotone} or for tackling soft constraints by a penalization approach, or  
{Stochastic Programming} problems, where the elements of the sum represent  a set of realizations of a random variable that can 
take a very large  set of potential finite values \cite{Nocedal,bollapragada2018adaptive}. Similar problems also arise in  {Sensor Networks}, where inference problems are solved in a distributed way, namely without gathering all the data in a single data center but storing part of the data in different databases \cite{bertsekas2011incremental}.

One of the most important applications which rely on the solution of very large finite-sum problems is the training phase of {Machine Learning} models (and the tools based on that) \cite{shalev2014understanding,Nocedal}
 where \eqref{eq: problem} represents the regularized empirical risk minimization problem, and the goal is to find the model parameters $w$ which minimizes the average loss on the $P$ observations.
The computational  per-iteration  burden needed to solve problem \eqref{eq: problem} depends on the size of the problem itself, namely both on the number of terms $P$ in the summation (which is large e.g. in big data setting) and on the size of the variable vector $n$ (which is large  e.g. in deep networks framework). 
Traditional  first or higher-order methods for smooth optimization can be naively applied to solve this problem \cite{NoceWrig06,bertsekasbook2016}. 
{However, these methods (a.k.a. \textit{batch methods}) usually use at each iteration all the $P$ samples to update the full vector $w$ and become too computationally expensive when $P$ is very large. Thus, }
there is an incentive to use less expensive per-iteration methods that exploit the structure of the objective function to reduce the computational burden meanwhile avoiding slowing down the convergence process.  

Widely used approaches for solving problem \eqref{eq: problem} are online or mini-batch gradient methods, which cycle through the component functions using a random or deterministic order and update the iterates using the gradient of a single or a small batch of component functions.  The reason for the success of online/mini-batch methods mainly relies on the different trade-offs in terms of per-iteration costs and expected per-iteration improvement in the objective function (see e.g. comments in \cite{Nocedal}).  An entire pass over all the samples $P$ is called an epoch, and solving problem \ref{eq: problem} usually requires tens to hundreds of epochs, even if a single epoch can be prohibitive when $P$ is huge (depending on available computational time).

These methods typically outperform full batch methods in numerical studies since each inner iteration makes reasonable
progress towards the solution {providing a good trade-off between per-iteration improvement and cost-per-iteration}, possibly enabling early stopping rules which are used to avoid overfitting in ML.

 Among the motivations for using such mini-batch methods, there is also  an implicit  randomness injection to the basic gradient iteration that can {lead to} better performance both in training and generalization. 
{Indeed, since the search direction need not be a descent direction, these methods possess an inherent non-monotonicity which allows them to explore larger regions thus, possibly escaping regions of attraction of inefficient local minimizers or saddle points.} Moreover, it has been  observed that the larger the batch size, the higher the chance to enter the basin of attraction of sharp minimizers where the function changes rapidly \cite{keskar2016large}, negatively impacting the ability of the trained model to generalize on new data \cite{hochreiter1997flat}.

The milestone methods in this class {of algorithms} are the Stochastic Gradient Descent (SGD), a.k.a. Robbins-Monro algorithm, \cite{Robbins&Monro:1951} which selects samples using a random with-replacement rule, and the Incremental Gradient algorithm (IG), which selects samples  following a cyclic deterministic order,  \cite{grippo1994class,luo1991convergence,bertsekas2011incremental}.
Finally, in practical implementations of mini-batch algorithms, a broadly used method is the Random Reshuffling algorithm (RR) which selects samples using a without-replacement sampling strategy, consisting in  shuffling the order of samples at each epoch, and processing that permutation by batches. RR often converges faster than SGD on many practical problems (see e.g. \cite{NEURIPS2020_c8cc6e90,bengio2012practical}).  Intuitively, without-replacement sampling strategies process data points more equally and are often easier and faster to implement, as it requires sequential data access, instead of random data access.  Empirical results show how RR methods  often lead to better performance than other mini-batch methods \cite{Bottou2012,shamir2016without, cannelli2020asynchronous}.

Most of the state-of-the-art methods are variants of these methods
\cite{Goodfellow-et-al-2016} that can be categorized mainly in \textit{gradient aggregation} and \textit{adaptive sampling} strategies. The former approach is based on improving the gradient estimation by considering the estimates computed in the previous iterations, such as in SVRG \cite{NIPS2013_4937}, SAG \cite{roux2012stochastic},  and Adam  \cite{Konur2013}. 
Gradient aggregation methods usually lead to good numerical performance, but their convergence has been proved in very few cases such as (strongly) convex functions, and usually require extra storage or full gradient computations (see e.g. IAG \cite{blatt2007convergent}, or SAG \cite{schmidt2017minimizing}), both limiting factors when $P$ is very large \cite{de2017automated}. Adaptive sampling strategies, instead, follow approaches that dynamically modify the number of terms $f_p$ used to estimate $\nabla f$ in order to improve the estimates. They require milder assumptions to converge, thus applying to a broader class of problems, but are less computationally efficient. Theoretically, they require the batch size to increase to infinity (see e.g. \cite{de2017automated,friedlander2012hybrid}) or the computation of exact quantities (e.g. the gradient $\nabla f$), or good estimates of them \cite{bollapragada2018adaptive,byrd2012sample}. 
 
By computing only approximations of $\nabla f$, mini-batch methods have worse convergence properties with convergence rates usually slower than full batch methods. Moreover, when it comes to the assumptions needed to prove convergence, mini-batch methods usually require {stronger} conditions to be satisfied with respect to full batch gradient methods.
Indeed, most of the full batch gradient methods only require Lipschitz-smoothness of each of the components $f_i$ that allow proving the Descent Lemma,  and in turn convergence for sufficiently small and fixed stepsize is proved.
Instead, when analyzing mini-batch methods, in an  either {probabilistic} or deterministic fashion, additional requirements are needed, and the stepsize must be driven to zero (which in turn allows driving the error to zero as well). The convergence analysis and related assumptions are quite different for random or deterministic sampling.

More specifically when analyzing the convergence and rate of SGD methods (sample with replacement), it is required either convexity and/or some kind of growth conditions on the second moments, such  e.g. the relaxed growth condition $\mathbb{E}(\|\nabla f_i(w)\|^2)  \le a+b \|\nabla f({w}) \|^2$ used in the analysis reported in \cite{Nocedal}.
In the recent paper \cite{khaled2020better}, the most common assumptions for proving convergence of SGD are reviewed and 
new  results have been proved using the 
expected smoothness assumption,  $\mathbb{E}(\|\nabla f_i(w)\|^2)  \le a+b \|\nabla f({w}) \|^2+c(f(w)-f_{inf})$, where $f_{inf}$ is a global lower bound on the function.
Expected smoothness is proved to be the weakest condition on the second moment among the existing ones and convergence of SGD is proved under the Polyak-Lojasiewicz condition (a generalization of strongly-convex functions).
In general, many papers have been devoted to study convergence and/or rate of convergence of modification of the basic SGD possibly weakening the required assumptions  (see e.g 
\cite{schmidt2017minimizing,NIPS2014_5258,gower2021sgd,gower2021stochastic,lei2019stochastic,chen2018lag,duchi2018stochastic,asi2019stochastic,nguyen2022finite,nguyen2021unified,Bottou10large-scalemachine,ghadimi2013stochastic}).

Regarding IG, convergence has been proved in \cite{bertsekas2011incremental} under the strong  assumption  $\|\nabla f_h({w}) \| \le a+b \|\nabla f({w}) \|$, for some positive constants $a,b$. This assumption holds for a few classes of functions, e.g. linear least squares \cite{Bertsekas2000}. Later, in \cite{solodov1998incremental} convergence is proved with a stepsize bounded away from zero, under the assumption that  $\|\nabla f_h(w)\|\le M$ for some positive constant $M$, which is a restrictive condition as it rules out strongly convex objectives. Finally, other deterministic schemes have been proposed (e.g. the Incremental Aggregated Gradient (IAG)
method for minimizing a finite sum of smooth functions where the sum is strongly convex \cite{gurbuzbalaban2017convergence}).

Proving convergence for the RR algorithm is even more difficult because without-replacement sampling introduces correlations and dependencies among the sampled gradients within an epoch that are harder to analyze, as discussed in \cite{gurbuzbalaban2021random,shamir2016without}. 
Recently, in \cite{NEURIPS2020_c8cc6e90} a radically different and fairly simpler analysis of RR has been developed that allows having a better insight into the behavior of the algorithm. They also discuss the convergence rate of the Shuffle-Once algorithm
\cite{nedic2001incremental} which shuffles the data only once before the training begins, and as a by-product obtained new results for IG. 
However, most of the convergence results of the RR algorithm and its variants require additional assumptions besides the $L$-smoothness,  such as (strong) convexity \cite{gurbuzbalaban2021random,shamir2016without,NEURIPS2020_c8cc6e90,malinovsky2021random,mishchenko2022proximal,sadiev2022federated},  some other kind of generalized convexity 
%such as the  Kurdyka-L{}ojasiewicz (KL) inequality
\cite{li2021convergence,ahn2020sgd} or conditions on the gradients of the single functions \cite{malinovsky2022server, NEURIPS2020_c8cc6e90}.
Quite recently in \cite{malinovsky2022server}, a new result on the convergence of a slight  modification of the basic RR method in the non-convex case has been proved. Their approach is closely related to our proposal in the fact that they prove effectiveness of making extra step along the basic RR direction. However, the value of the Lipschitz constant is used.

Therefore, although online methods are broadly applied to solve large finite-sum problems, especially in the machine learning field, convergence analysis for such algorithms in the non-convex setting is still far from satisfactory.
 
Another critical issue in this class of methods is the practical choice of the decreasing rule for the learning rate.
In practice, a preferred schedule is chosen manually by testing many different schedules in advance and choosing the one leading to the smallest training or generalization error. This process can be computationally heavy and can become prohibitively expensive in terms of time and computational resources incurred \cite{ward2020adagrad}.
Adaptive  gradient methods such as AdaGrad \cite{Duchi:2011:ASM:1953048.2021068,ward2020adagrad}, update the stepsizes on-the-fly according to the gradients received along the iterations. Recently, in \cite{ward2020adagrad} convergence of the Norm-AdaGrad has been proved without convexity  but requiring the strong assumption that $\|\nabla F(w)\|^2\le \gamma^2$ uniformly.

Therefore, it is evident how the trade-off between theoretical properties and cost per iteration plays a crucial role in the definition of optimization methods for large-scale finite-sum problems. Establishing convergence of an algorithm with mild assumptions requires a computational overhead to determine exact quantities or good estimations with low variance. On the other hand, the cost per iteration of the algorithm can be reduced at the extra-cost of requiring more stringent assumptions on the objective function and, thus, reducing the range of theoretical applicability of the algorithm. 

\section{Assumptions and summary of contributions}\label{sec:motiva}

In this paper, we focus on ease-controlled modifications of Random Reshuffling gradient (RR) methods, including the case of the IG method. We aim to define schemes that converge under mild assumptions, and in particular, without requiring either convexity assumptions (or similar-like) on the functions or growth/bounded conditions on the gradients of the single components $f_p$.

In particular, throughout the paper, we require the following two assumptions to hold.

\begin{assumption}\label{ass:compact}
The function $f$ is coercive, i.e. all the level sets 
\[
  {\cal L}(w_0) = \{w\in\Re^n:\ f(w) \leq f(w_0)\}
\]
are compact for any $w_0\in\Re^n$.
\end{assumption}

\begin{assumption}\label{ass:lipschtizagradient}
The gradients $\nabla f_p$ are Lipschitz continuous for each $p=1,\dots, P$, namely there exists $L>0$ such that
\[
 \|\nabla f_p(u)-\nabla f_p(v)\|\le L\|u-v\|\qquad \forall \ u,w\in\R^n
\]
\end{assumption}

Assumption \ref{ass:compact} enforces the existence of a global solution.  It is trivially satisfied e.g. in training problems for Deep Neural Networks (DNNs) when using a $\ell_2$  regularization term for the empirical error. 
Assumption \ref{ass:lipschtizagradient} is a standard and mild assumption in gradient-based methods.

The idea at the basis of the proposed schemes, firstly proposed in \cite{secciaPhD2020}, consists in perturbing as less as possible the basic iteration of the RR/IG gradient algorithm to retain their light computational effort and the low memory requirement per iteration. 
The ingredients used are indeed the basic mini-batch iteration and tools derived from derivative-free methods. 

In particular, we define two algorithmic schemes in which the IG/RR iteration is controlled by using a watchdog rule and a derivative-free linesearch that activates only sporadically to guarantee convergence. The two schemes differ in the watchdog and the linesearch procedures, which are performed using either a monotonic or a non-monotonic rule.
The two schemes also allow controlling the updating of the learning rate used in the main IG/RR iteration, avoiding the use of pre-set rules that may drive it to zero too fast,
thus overcoming another challenging aspect in implementing online methods, which is the updating rule of the stepsize.
We prove convergence under the mild assumption of Lipschitz continuity of the gradients of the component functions.

To this purpose, we introduce a Controlled mini-batch Algorithm (\CMA) whose basic idea is to control both the iteration and the stepsize of RR mini-batch gradient iteration. Indeed, a  trial point $\widetilde w^k$ is obtained by using a  mini-batch  gradient  method with a fixed stepsize. The trial point can be either accepted when a \textit{watchdog}-flavoured condition \cite{chamberlain1982watchdog} is verified and a sufficient decrease of the objective function is guaranteed. When this is not satisfied, we resort to a linesearch procedure to enforce convergence. However, 
since we want to avoid computing the gradient even periodically,
%\change[RS]{since the gradient is not available and its calculation is not viable not even periodically}{to avoid computing the gradient of the objective function}
%\todo{calcolare il gradiente con AD costa al Massimo 6 volte un calcolo di funzione obiettivo, quindi quelo scritto prima non mi sembrava una buona motivazione},
we adopt a  derivative-free-type linesearch  \cite{grippo1988global,lemarechal1981view,lucidi2002global}, which is used to define the new iteration and to possibly update the stepsize used in the mini-batch cycle. 
Therefore, the extra cost to pay for lightening the assumptions needed to prove convergence is given by the function evaluation needed at the end of each epoch and 
possibly in the linesearch.

%\todo[inline]{una fuzione la calcoliamo sempre a fine di ogni epoca}
%\remove{In this work, we introduce a novel optimization approach for solving problem eqref{eq:problem} called Controlled mini-batch Algorithm (\CMA), propose two \CMA-based frameworks, called \CMA\ and \NMCMA, and prove convergence towards stationary points under very mild assumptions. \CMA\ methods combine some of the theoretical properties of batch methods and the good numerical performance of mini-batch methods, by requiring very mild and standard assumptions and yet allowing to apply mini-batch iterations within the framework.}
The Controlled mini-batch Algorithm (\CMA) is proposed in two versions, a monotone and non-monotone version which are called \CMA\ and \NMCMA\ respectively. The two schemes require different watchdog and linesearch updating rules and are not comparable from a theoretical point of view in the final effort required. 
%\todo{possiamo forse dire perche introduciamo due versioni dell'algoritmo. Quali soon I vantaggi di uno e dell'altro?}

In order to assess the performance of the proposed methods, numerical results have been carried out by solving the optimization problem behind the learning phase of Deep Neural Networks. We perform  an extensive numerical experiments using different Deep  Neural Architectures on a benchmark of varying size datasets. We compare our implementation with both a full batch gradient method (i.e. L-BFGS) and a standard implementation of IG/RR online methods, proving that the computational effort is comparable with the corresponding online methods and that the control on the learning rate may allow a faster decrease in the objective function.  The numerical results on standard test problems suggest the effectiveness of \CMA-based methods in leading to better performance results than state-of-the-art batch methods (i.e. L-BFGS) and standard mini-batch methods (i.e. IG), by finding better solutions with smaller generalization errors when dealing with large problems. However, although from the  computational results reported in this paper \NMCMA\ slightly outperforms \CMA\, from a theoretical perspective it is not possible to claim a superior performance of \CMA\ over \NMCMA\ or vice versa.

The remainder of the paper is organized as follows. In  Section \ref{sec:motiva} we present the assumptions used in the paper and our main contributions. In Section \ref{sec:preliminaries} we review the properties of the basic mini-batch iteration. In  Section \ref{sec:CMA} we first provide the idea behind Controlled mini-batch Algorithms, then we introduce the two frameworks, a monotone version \CMA\ and a non-monotone version \NMCMA\, and for each of them provide convergence results under very mild standard assumptions. 
In section \ref{sec:CMA_numerical} the effectiveness of two methods is tested % by showing numerical results on several test problems against L-BFGS and a standard IG algorithm
.  Finally, in section \ref{sec:CMA_conclusion} conclusions on the novel  optimization approach are drawn.

\section{A unified framework for mini-batch gradient algorithms and basic convergence properties}\label{sec:preliminaries}
Given the problem \eqref{eq: problem}, a Mini-batch Gradient (MG) method updates the current iteration by using a small batch of the samples {${ I}_p\subset\{1,\dots,P\}$}, that for the sake of simplicity and without loss of generality will be shortened up as a single term $p$, thus we denote 
$$\ds f_{p}(w)=\sum_{h_i\in I_p} f_{h_i}(w)\text{ and }\ds \nabla f_{p}(w)=\sum_{h_i\in I_p} \nabla f_{h_i}(w)$$

An outer iteration $w^k$ of the method  is obtained after a {full} cycle over $P$ inner iterations,  which is called an \emph{epoch}. Starting with $\widetilde w_{0}\in \R^n$ the inner iterations updates the point as 

\begin{equation}\label{eq:basicIG}
\widetilde w_{i}= \widetilde w_{i-1} - {\zeta} \nabla f_{h_i}(\widetilde w_{i-1})\qquad i=1,\dots, P
\end{equation}
where $h_1 \dots h_P$ represent  a permutation of the indexes $\{1,\dots,P\}$ and $\zeta>0$ is the stepsize (or learning rate) which we consider  constant over the epoch.

Each inner step of the method \eqref{eq:basicIG} is thus very cheap, involving only the computation of the gradient 
$\nabla f_{h_i}(\widetilde w_{i-1})$
corresponding
to a mini-batch of samples.

We first observe that
we have, for all $i=1,\dots,P$,
\begin{equation}\label{eq:basicIG2}
\widetilde w_i = \widetilde w_0-\zeta \sum_{j=1}^i \nabla f_{h_j}(\widetilde w_{j-1}). 
\end{equation}

In 
Incremental Gradient  \cite{bertsekas2011incremental,Bertsekas2000} and Random Reshuffling methods \cite{gurbuzbalaban2021random}, the sequence $\{w^k\}$  is obtained as 
$w^k=\widetilde w_P^{k-1}$ for $k=1,2,\dots$ where $\widetilde w^k_P$ is obtained as in \eqref{eq:basicIG} starting with $\widetilde w^k_0=w^k$
so that for each $k=1,2,\dots $ at the end of each epoch we obtain
$$\widetilde w^k_P = \widetilde w^k_0-\zeta \sum_{j=1}^P \nabla f_{h_j}(\widetilde w^k_{j-1}), $$
where in
IG or Shuffle-Once
samples are always chosen in the same order (which is fixed in a deterministic or random way respectively)  per epoch, e.g. $h_1=1, \dots, h_P=P$ for all $k$, whereas in RR methods $h_1, \dots ,h_P$ are 
{picked via a sampling without replacement from the set of $P$ terms available.}

Convergence properties of these methods are studied with respect to the behaviour of the outer sequence $\{w^k\}$. 

In this section, 
we aim to derive properties of the sequence $\{w_i^k\}$  generated by  \eqref{eq:basicIG} when it is encompassed in a generic outer scheme where a sequence $\{w^k\}\subseteq \R^n$ and a sequence of nonnegative scalars $\zeta^k$ are given and each element $w^k$ of the sequence represents the starting point $\widetilde w^k_{0}$ of the inner MG iterations.
Thus, we consider the simple iterative scheme reported in Algorithm \ref{alg:IG}

\begin{algorithm}[htb]
\begin{algorithmic}[1]
    \State Given the sequences $\{\zeta^k\}$, and $\{w^k\}$ with $\zeta^k \in \R_+$ and $w^k\in\R^n$
    \For {$k=0,1,2\dots $}       
    \State Compute $(\widetilde w^k, d^k)$ = \texttt{Inner\_Cycle}($w^k,\zeta^k$)
    \EndFor
\end{algorithmic}
\caption{mini-batch Gradient (MG) iteration}
\label{alg:IG}
\end{algorithm}

\begin{algorithm}[htb]
\begin{algorithmic}[1]
    \State {\bf Input}: $w^k$, $\zeta^k$
    \State Set $\widetilde w_0^k = w^k$
    \State Let $I^k=\{h_1^k,\dots ,h_P^k\}$ be a permutation of $\{1,\dots ,P\}$
    \For {$i =1,...,P$}
        \State $\tilde d_i^k = -\nabla f_{h_i^k}(\widetilde w_{i-1}^k)$
        \State $\widetilde w_{i}^k= \widetilde w_{i-1}^k + {\zeta^k} \tilde d_i^k$
    \EndFor
    \State {\bf Output} $\widetilde w^k = \widetilde w^k_P$ and $d^k = \displaystyle\sum_{i=1}^P\tilde d_i^k$
\end{algorithmic}
\caption{\texttt{Inner\_Cycle}}
\label{alg:innercycle}
\end{algorithm}

The \texttt{Inner\_Cycle} consists in iteration \eqref{eq:basicIG}, applied with starting point $\widetilde w_0^k = w^k$ and returns  $\widetilde w^k=\widetilde w^k_P$ and $\displaystyle d^k = -\displaystyle\sum_{i=1}^P\nabla f_{h_i^k}(\widetilde w_{i-1}^k)$
obtained {at the end of each epoch.}

Depending on the assumptions on the sequences $\{w^k\}$ and $\{\zeta^k\}$ we can prove different properties satisfied by the sequences
$\{\widetilde w^k\}$, $\{\widetilde w_{i}^k\}$ for $i=1,\dots, P$, and  $\{d^k\} $ generated by the overall scheme.

{Under Assumption \ref{ass:lipschtizagradient}, iteration \eqref{eq:basicIG} has the following property,} whose proof is reported in the Appendix.

\begin{lemma}\label{prop:Lipschitz}
Let Assumption \ref{ass:lipschtizagradient} hold and let $\widetilde w_i$ be defined by iteration \eqref{eq:basicIG}.
Then we have, for all $i=1,\dots,P$,
\begin{eqnarray}\label{boundnorm_i_noK_2}
\|\widetilde w_i - \widetilde w_{0}\|  & \leq &  \zeta\left(  
L\sum_{j=1}^{i}\|\widetilde w_{j-1} - \widetilde w_0\| +\sum_{j=1}^{i}  \|\nabla f_{h_j}(\widetilde w_0)\|\right).
% \|\widetilde w_i - \widetilde w_{0}\|  & \leq &  \zeta\left(  
% L\sum_{j=1}^{i}\|\widetilde w_{j-1} - \widetilde w_0\| +\sum_{j=1}^{i}  \|\nabla f_j(\widetilde w_0)\|\right).
\end{eqnarray}
\end{lemma}

We start by proving that when both $\{w^k\}$ and $\{\zeta^k\}$ are bounded,  then the sequences $\{\widetilde w_i^k\}$ are bounded too for all $i=1,\dots, P$.

% \todo{vedi se già dimostratro in Nocedal curtis ecc} RS: NO LORO SI BASANO SOLO SULLE RIDUZIONE DI F
\begin{proposition}
Let Assumption \ref{ass:lipschtizagradient} hold and let $\{w^k\}$ and $\{\zeta^k\}$ be bounded sequences of points and positive scalars. 
Let $\{\widetilde w_i^k\}$ for $i=1,\dots,P$ be the sequences generated by Algorithm \ref{alg:IG}.

Then, the sequences $\{\widetilde w_i^k- w^k\}$, for $i=1,\dots,P$, are all bounded.
\end{proposition}
\begin{proof} The proof is by induction. 
Using the definition of $\widetilde w_i^k$ as in Step 5 and Step 6 of Algorithm \ref{alg:innercycle} we can write for $i=1$
%Recalling (\ref{eq:basicIG}), we have that
\[
\|\widetilde w_1^k - \widetilde w_0^k\| = \|\widetilde w_1^k - w^k\| = \zeta^k\|\nabla f_{h^k_1}(w^k)\|
\]
which, considering the boundedness of $w^k$, $\zeta^k$ and continuity of $\nabla f_{h^k_1}$, proves that $\{\widetilde w_1^k - \widetilde w_0^k\}$ is bounded. 

%\remove{\par\noindent Now, we prove that $\{\widetilde w_i^k- w^k\}$ is bounded by induction on $i$. }

Thus, we assume that $\{\widetilde w_j^k- w^k\}$ are bounded for all $j=1,\dots,i-1$, and prove that $\{\widetilde w_i^k- w^k\}$ is bounded as well. In fact, thanks to Lemma \ref{prop:Lipschitz} we can write
$$
\|\widetilde w^k_i -  w^k\|   \leq   \zeta^k\left(  
L\sum_{j=1}^{i}\|\widetilde w^k_{j-1} - w^k\| +\sum_{j=1}^{i}  \|\nabla f_{h^k_j}(w^k)\|\right),
$$
which, by the induction assumption, boundedness of $ w^k$ and $\zeta^k$ and the continuity of each $\nabla f_i$, proves boundedness of $\{\widetilde w^k_i -  w^k\}$ for all $i$.
\end{proof}

{We now prove that when the stepsize $\zeta^k$ is driven to zero,
%\todo{LP RISPOSTA: NO, La sequenza $\zeta^k$ è data. Attenzione, in MG non si ci sta nessun passo per aggiornare $\zeta$ (va aggiunto?!)}
the bounded sequences $\{\widetilde w^k_i \}$ converge to the same limit point of the outer sequence $\{w^k\}$ and that the sequence of directions $\{d^k\}$ converges to   
$-\nabla f(\bar w)$. 
}
{This result is at the basis of the convergence of IG and RR methods.}
\begin{proposition}\label{limtildew_general}\label{limtildew_CMA2}\label{limtildew_NMCMA}
Let Assumption \ref{ass:lipschtizagradient} hold. Assume that $\{w^k\}$ is bounded and that $\lim_{k\to\infty}\zeta^k = 0$. Let $\{\widetilde w^k_i\}$ and $\{d^k\}$ be the sequences of points and directions produced by Algorithm \ref{alg:IG}. 

Then, for any limit point $\Bar w$ of $\{w^k\}$ a subset of indices $K$ exists such that 
    \begin{eqnarray}
        && \lim_{k\to \infty,k\in K} w^k=\Bar w,\label{assert_C}\\
        && \lim_{k\to \infty,k\in K} \zeta^k=0,\label{assert_D}\\
        && \lim_{k\to\infty,k\in K}\widetilde w^k_i =\bar w,\quad\mbox{for all $i=1,\dots, P$}\label{assert_A_0}\\
        && \lim_{k\to\infty,k\in K} d_k = -\nabla f(\bar w).\label{assert_B_0}
    \end{eqnarray}
%\todo[inline]{Fare attenzione al sottoinsieme di indici $K$}
\end{proposition}
\begin{proof}
Let us consider any limit point $\bar w$ of $\{w^k\}$. Then, recalling that $\lim_{k\to\infty}\zeta^k = 0$, we have that a subset of indices $\bar K$ exists such that
\begin{eqnarray*}
        && \lim_{k\to \infty,k\in\bar K} w^k=\Bar w,\\
        && \lim_{k\to \infty,k\in\bar K} \zeta^k=0.
\end{eqnarray*}
Now, for every iteration index $k$, let $I^k$ denote a permutation of the set of indices $\{1,\dots,P\}$. Since $P$ is  finite,  the number of such permutations is finite. Then, we can extract an infinite subset of indices $K\subseteq\bar K$ such that
\begin{eqnarray}
&& \lim_{k\to\infty, k\in K} w^k = \bar w\label{assert_C_1}\\
&& \lim_{k\to\infty, k\in K} \zeta^k = 0 \label{assert_D_1}\\
&& I_k = \bar I= \{\bar h_1,\dots,\bar h_P\},\ \forall\ k\in K\nonumber
\end{eqnarray}
Thus, (\ref{assert_C_1}) and (\ref{assert_D_1}) prove, respectively, (\ref{assert_C}) and (\ref{assert_D}).
From the definition of Algorithm \ref{alg:IG}, and applying Proposition \ref{prop:Lipschitz} we have that for $i=1,\dots,P$ and $k\in K$, we can write
\begin{eqnarray}\label{boundnorm_i_0}
\|\widetilde w^k_i -  w^k\|   \leq   \zeta^k\left(  
L\sum_{j=1}^{i}\|\widetilde w^k_{j-1} - w^k\| +\sum_{j=1}^{i}  \|\nabla f_{\bar h_j}(w^k)\|\right).
\end{eqnarray}

Now, to prove (\ref{assert_A_0}), we proceed by induction on $i$. Consider first $i=1$
so that 
$$\|\widetilde w^k_1 -  w^k\|   \leq    \zeta^k \|\nabla f_{\bar h_1}(w^k)\|$$
Taking the limit for $k\to\infty$, $k\in K$ %in (\ref{boundnorm_10_0}),
recalling that $\{w^k\}$ is bounded and continuity of $\nabla f_{\bar h_1}$, we obtain
\[
\lim_{k\to\infty, k\in K}\|\widetilde w^k_1 - w^k\| = 0
\]
which means, recalling (\ref{assert_C_1}), that
\begin{equation}\label{parziale_1_0}
\lim_{k\to\infty, k\in K}\widetilde w^k_1 = \bar w.
\end{equation}
%Thus, \eqref{assert_A_0} holds for $i=1$.
%Now, to prove (\ref{assert_A_0}), we proceed by induction on $i$.
By induction on $i$,  let us assume that, for all  $j= 1,\dots,i-1$,  
\begin{equation}\label{assunzione_induz_0}
\lim_{k\to\infty,k\in K}\|\widetilde w^k_{j} - w^k\| = 0%\\
%&& \lim_{k\to\infty, k\in K} \widetilde w^k_{i} = \bar w
\end{equation}
% and show that
% \[
% \lim_{k\to\infty,k\in K}\|\widetilde w^k_{i} - w^k\| = 0
% \]
% as well. 
By taking the limit for $k\to\infty$, $k\in K$, in (\ref{boundnorm_i_0}), and recalling the induction assumption (\ref{assunzione_induz_0}), boundedness of $\{w^k\}$ and continuity of $\nabla f_{\bar h_j}$ for all $j$, we have that
\[
\lim_{k\to\infty,k\in  K}\|\widetilde w^k_{i} - w^k\| = 0
\]
which means, recalling (\ref{assert_C_1}), that
% \begin{equation}\label{passo_induttivo}
% \lim_{k\to\infty,k\in K}\widetilde w^k_{P} = \bar w.
% \end{equation}
\begin{equation}\label{passo_induttivo}
\lim_{k\to\infty,k\in K}\widetilde w^k_{i} = \bar w.
\end{equation}
Then, (\ref{parziale_1_0}) and (\ref{passo_induttivo}) prove (\ref{assert_A_0}). 

In order to prove (\ref{assert_B_0}), let us recall the definition of $d^k$, for $k\in K$, i.e.
\[
d^k = -\sum_{i=1}^P\nabla f_{\bar h_i}(\widetilde w^k_{i-1}).
\]
Then, (\ref{assert_B_0}) follows by recalling (\ref{assert_A_0}) and the continuity of $\nabla f_{\bar h_i}$, for $i=1,\dots,P$.
\end{proof}

\begin{remark}
In Algorithm \ref{alg:innercycle} the \emph{for cycle} at Step 4 can be modified 
using in place of $P$ a value $P^k\le P$ for all $k$ and such that  $P^k=P$ for infinitely many iterations. Indeed,  under this condition we have that Proposition \ref{limtildew_general} still holds on the further subsequence $\{k:\ P^k=P\}$.
\end{remark}

Proposition \ref{limtildew_general} shows that, given an ``external'' converging sequence  $\{w^k\}$ and driving the stepsize $\zeta^k\to 0$, the ``attached'' sequences $\widetilde w_i^k$ for $k\in K$ remains in a neighborhood of $w^k$ and converge to the same limit point $\bar w$. 
Thus it proves that the direction $d^k$ used for updating the point $w^k$ tends in the limit to be the gradient itself. 
This implies that the error  $$\left\|\nabla f(w^k)- \sum_{i=1}^P\nabla f_{h_i}(\widetilde w^k_{i-1}) \right\|$$
is going to zero in the limit.

The error going to zero has been exploited to prove convergence of  deterministic IG and RR methods requiring strong assumptions on the objective functions as done in most papers \cite{Bertsekas2000,palagisecciaSOCO2021,gurbuzbalaban2017convergence,schmidt2017minimizing,Bottou10large-scalemachine}. 
Indeed, these strong assumptions are not satisfied in many important machine learning models.

In Proposition \ref{limtildew_general}, these assumptions on the objective function are replaced by a strong assumption on the sequence $\{w^k\}$, which has to remain bounded. {In the next sections, we show how to enforce this assumption by embedding Algorithm \ref{alg:innercycle} in a more specific framework producing the sequence $\{w^k\}$.
}

Another crucial aspect in methods using diminishing stepsize rule is how to drive the stepsize $\zeta^k\to 0$. Indeed many heuristic rules have been defined \cite{Goodfellow-et-al-2016}  to avoid the stepsize decreasing too fast and thus slowing down convergence.
{The algorithms that we define in the next sections also allow tackling this aspect in an automatic way by controlling the reduction of the stepsize. In particular, the price to pay for lightening the assumptions requires a limited additional computational effort by checking  periodically a sufficient reduction of the objective function. }

\section{Controlled mini-batch Algorithms} \label{sec:CMA}
In this section, we  introduce the idea behind Controlled mini-batch algorithms. We define two algorithm schemes, the monotone \CMA\ and the non-monotone \NMCMA, and for each of them we prove convergence properties under specific light assumptions. 
{Indeed, the aim of \CMA-based algorithms is to define schemes that require a low increased computational cost with respect to traditional RR or IG algorithms, and allow to control convergence without requiring strong assumptions on the objective function. In particular, neither convexity, }nor strong growth conditions on the gradients of each $f_p$ and on the full gradient, as in \cite{bertsekas2011incremental,palagisecciaSOCO2021,gurbuzbalaban2021random} are required, {but}
{we will use the only Assumptions \ref{ass:compact} (assumed always satisfied) and \ref{ass:lipschtizagradient} (recalled when needed). }

The basic idea behind  the Controlled mini-batch Algorithm schemes consists {in defining the ``external'' sequence  $\{w^k\}$ by} trying to accept as much as possible the points generated by the MG algorithm  as defined in \ref{alg:innercycle}, namely 
setting $w^{k+1}={ \widetilde w_P^k}$ for \emph{as many iterations $k$ as possible}. The acceptance of the point ${ \widetilde w_P^k}$ is controlled  by means of both a watchdog approach and, in extreme cases, a derivative-free linesearch procedure involving a few evaluations of just the objective function. 

{In particular, the outer iterations $w^k$ %which represent the sequence of estimates t the stationary point 
}
{is defined as 
\begin{equation}\label{eq:updategeneral}
w^{k+1}=w^k+\alpha^k d^k
\end{equation}
}

where  
$\displaystyle d^k = -\displaystyle\sum_{i=1}^P\nabla f_{h_i^k}(\widetilde w_{i-1}^k)$ 
is the direction obtained in the \texttt{Inner\_Cycle}  
and the stepsize can take the values

$$\alpha^k=\begin{cases}
          \zeta^k & \text{the trial point $\widetilde w_P^k$ is accepted  } \\
          0 & \text{restart from the last point $w^k$} \\
          >\zeta^k   & \text{take a longer step along $d^k$} \\
        \end{cases}$$

In Figure \ref{fig:scheme} a picture of  the outer iterations in the MG gradient method (\texttt{Inner cycle}) and in the Controlled Mini Batch scheme is reported.

We highlight how the outer iteration in Controlled mini-batch Algorithms  allows an extrapolation step 
along $d^k$ but also the possibility of having a restart, thus discarding completely the iterations done in the last epoch. This kind of control allows proving convergence under very mild assumptions. The computational experience reported in section \ref{sec:CMA_numerical} shows that the restart step appears very seldom and does not affect the overall performance.
Quite recently, in \cite{malinovsky2022server}, a result on  the convergence of a modification of the RR method which considers a similar {extrapolation}  
step along the direction $d^k$ has been provided. The framework presented there is more general, as it considers federated learning with $M$ nodes. When $M=1$, it is a RR method for problem \eqref{eq: problem}. They proved that the extra step improves on RR in a very clear and clean manner. However, the main difference between their and our approach relies on the hypothesis needed to prove convergence: we do not need to know to any extent the Lipschitz constant, so that $\zeta^k $ is adjusted in an automatic way. Moreover, we also include the possibility of restarting the iteration.

\begin{figure}
    \centering     %%% not \center
        \subfigure[MG]
            {\label{fig:MG}
            \includegraphics[width=0.48\textwidth]{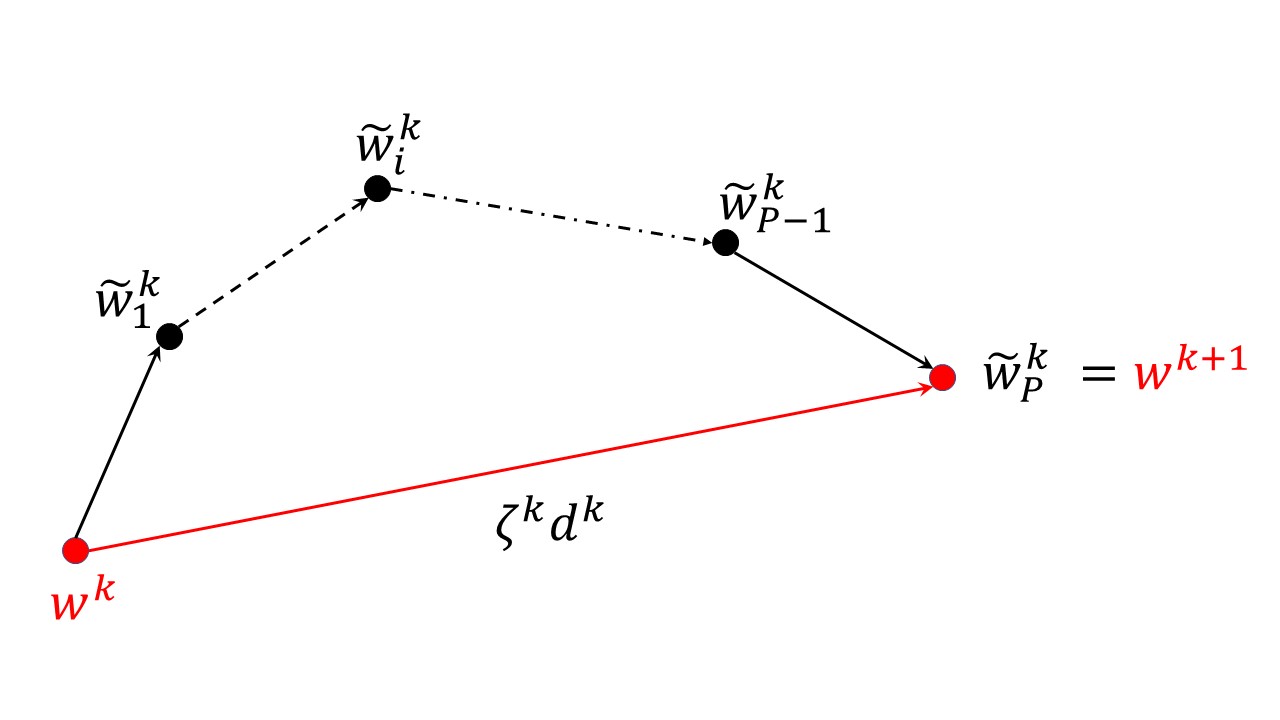}}
        \subfigure[CMA]
            {\label{fig:CMA}
            \includegraphics[width=0.48\textwidth]{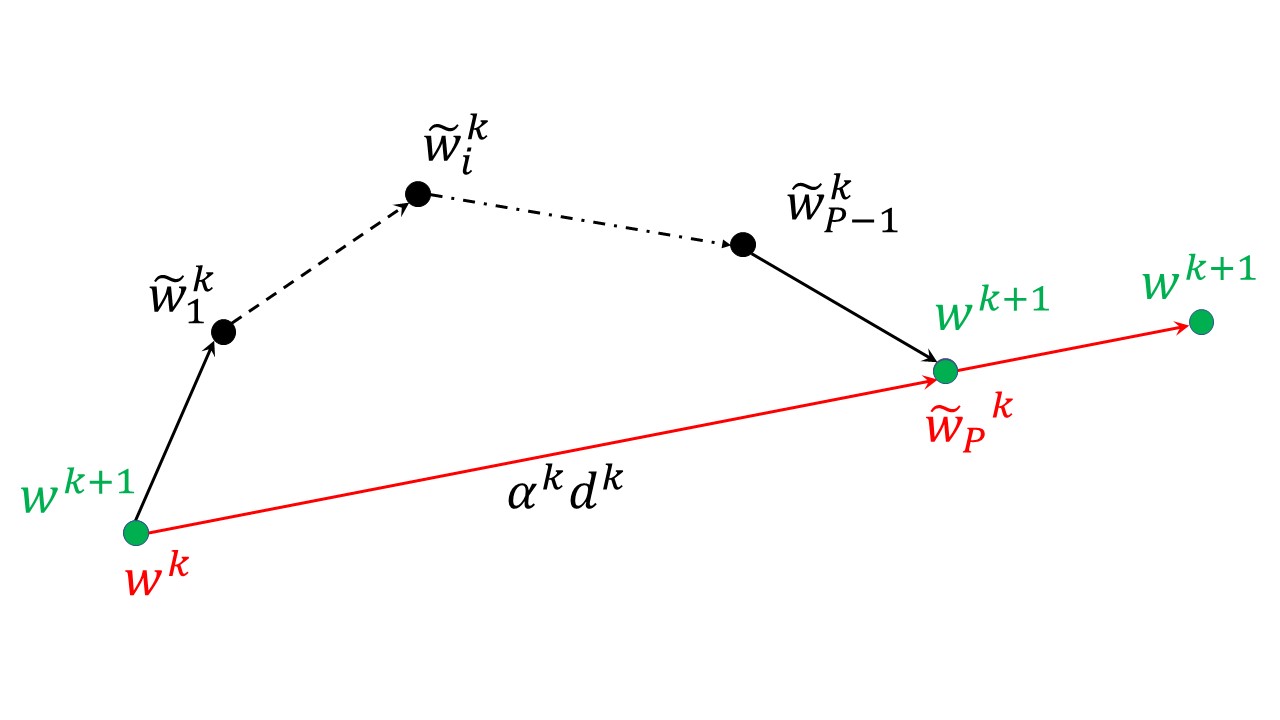}}
    \caption{a) MG gradient outer iteration; b) Controlled mini batch outer iteration: $w^{k+1}$ is selected as on one of the green points.}
    \label{fig:scheme}
\end{figure}

Further,  we also have the chance to control the stepsize $\zeta^k$ used in the \texttt{Inner\_Cycle}.
 Algorithm \ref{alg:innercycle} constitutes the  
inner iterations which defines trial points $\widetilde w^k$ and search directions $d^k$.
 Eventually, the stepsize $\zeta^k$ used in in the \texttt{Inner\_Cycle} is driven to zero. 
However, the decreasing rule depends on watchdog rule on the iteration, and it may happen that $\zeta^{k+1}=\zeta^k$. This avoid the inner stepsize to decrease too fast and slowing down the practical convergence.

We call the method ``Controlled mini-batch Algorithm'' since at each {``outer''} iteration the watchdog approach is leveraged to apply mini-batch iterations as much as possible, while reductions of the objective function over the iterations and of  the stepsize $\zeta^k$ are checked and in case {enforced via a DFLS procedure.}

In the following, two different versions of Controlled mini-batch Algorithms  are introduced that use either a monotone or non-monotone linesearch procedures thus forcing either a monotone or nonmononotone decreasing sequence $\{f(w^k)\}$. Both the two DFLS procedures use as starting stepsize the value of $\zeta^k$ and envisage the possibility of extrapolating, namely of taking larger steps along the direction $d^k$. For the sake of simplicity, we will refer to these two methods as \CMA\ and \NMCMA\ respectively. 

The two methods {differ not only in the DFLS but also in the acceptance rules of the trial points and in the reduction of the stepsize. From a theoretical point of view we cannot prove the superiority of one w.r.t. the other, while, from a practical perspective, they can} explore different regions of the searching space and can achieve different numerical performances.

\subsection{Monotone Controlled mini-batch Algorithm}

In this section, we present the  Controlled mini-batch Algorithm (\CMA) which is  a \emph{monotonically} decreasing method. 
Following the general description, a trial point $\widetilde w^k$ and a direction $d^k$ are obtained by applying Algorithm \ref{alg:innercycle}  with a constant stepsize $\zeta^k$.

\CMA\  checks 
if a sufficient reduction in the trial point  is satisfied, namely
$f(\widetilde w^k) - f(w^k) \leq - \gamma\zeta^k$. In this case,
the trial point is accepted, and the stepsize $\zeta^k$ is left unchanged. 

Otherwise, \CMA\  checks whether the direction $d^k$ can be considered as a good estimate of a descent direction or if  we are close enough to convergence. 
In these cases, the inner stepsize $\zeta^k$ can be reduced, and  the direction can possibly be used
within an Extrapolation Derivative-Free Linesearch (EDFL). When used, EDFL computes the outer stepsize $\alpha^k$  used to update the iteration. 
It checks first for a sufficient monotonic reduction using a derivative-free condition
$f(\widetilde w^k)-f(w^k)\le -\gamma\zeta^k\|d^k\|^2$, and 
when satisfied, it employs an extrapolation procedure for taking longer steps than $\zeta^k$ along the direction.
The scheme of the Linesearch procedure EDFL is provided in Algorithm \ref{alg:EDFL2} \cite{lucidi2002global}.

\begin{algorithm}[htb]
\caption{Extrapolation Derivative-Free Linesearch (EDFL)}
\label{alg:EDFL2}
\begin{algorithmic}[1]
\State Input $(w^k,d^k,\zeta^k;\gamma,\delta)$: $w^k\in \R^n,d^k\in \R^n,\zeta^k>0$, $\gamma \in (0,1), \delta\in(0,1)$
\State Set $\alpha=\zeta^k$
\If {$f(w^k +\alpha d^k)> {f(w^k)}-\gamma\alpha\|d^k\|^2$}
    \label{edfl:step4}    \State Set $\alpha^k = 0$ and \textbf{return} %$\alpha^k$
\EndIf
\While {$f(w^k +(\alpha/\delta) d^k)\leq \min\{f(w^k)-\gamma(\alpha/\delta)\|d^k\|^2, f(w^k+\alpha d^k)\}$}
     \label{edfl:step7}   \State Set $\alpha=\alpha/\delta$
\EndWhile
\State Set $\alpha^k = \alpha$ and \textbf{return} 
\State Output: $\alpha^k$
 \end{algorithmic}
 \end{algorithm}

The scheme of the \CMA\ is provided in Algorithm \ref{alg:CMA2}.
We note that the stepsize returned by the EDFL is called $\widetilde \alpha^k$ and it is used together with $\zeta^k$ and $d^k$ to check several different conditions on the sufficient reduction of the objective function. The final stepsize $\alpha^k$ is set to zero only when all the conditions are not met, and the produced point would be outside the Level set ${\cal L}(w^0)$. In this way we try to allow acceptance of the trial point as much as possible, thus perturbing the classical RR or IG iteration as few times as possible.
Also,
\CMA\ allows a further degree of freedom in choosing 
$w^{k+1}$ w.r.t. to \eqref{eq:updategeneral} that can actually be set to any point such that 
    $f(w^{k+1})\leq f( w^k + \alpha^k d^k)$

 \begin{algorithm}[htb]
\begin{algorithmic}[1]
    \State Set $ \zeta^0 >0, \theta\in(0,1), \tau > 0$, $\gamma\in(0,1), \delta\in(0,1)$
    \State Let $w^0\in \R^n,k=0$
    \For {$k=0,1,2\dots $}
        \State Compute $(\widetilde w^k, d^k)$ = \texttt{Inner\_Cycle}($w^k,\zeta^k$)
   \If{{ $f(\widetilde w^k) \leq f(w^k) - \gamma\zeta^k$}} \label{cma:step10} 
    \State {Set $\zeta^{k+1} = \zeta^k$}  {and $\alpha^k = \zeta^k$} \label{cma2:step11}
    \Else
    \If{$\|d^k\| \leq \tau\zeta^k$}
        \State \label{cma2:step9} Set $\zeta^{k+1}=\theta \zeta^k$
        and {$\alpha^k=\begin{cases}
         \zeta^k & \text{if } f(\widetilde w^k)\le f(w^0)\\
        0 & \text{otherwise}\\
        \end{cases}$}
        
    \Else
         \State $\widetilde\alpha^k =EDFL(w^k,d^k,\zeta^k;\gamma,\delta)$ %change Laura 23 ottobre 2022
         \If {$\widetilde\alpha^k\|d^k\|^2 \leq\tau\zeta^k$}
            \State \label{step19_CMA2}Set $\zeta^{k+1}=\theta \zeta^k$  and {$\alpha^k=\begin{cases} %\textcolor{dgreen}
          \widetilde\alpha^k & \text{if } \widetilde\alpha^k > 0\\
          \zeta^k & \text{if } \widetilde\alpha^k=0 \text{ and }f(\widetilde w^k)\le f(w^0)\\
          0 & \text{otherwise}\\
        \end{cases}$}
                   
         \Else
            \State Set $\zeta^{k+1}=\zeta^k,$ \label{step18_CMA2}
            and $\alpha^k = \widetilde\alpha^k$
         \EndIf
    \EndIf
    \EndIf
    \State Set $y^k = w^k + \alpha^k d^k$ 
    \State Select $ w^{k+1}$ such that $f(w^{k+1})\leq f( y^k)$ \label{step20_CMA2}
   \EndFor
\end{algorithmic}
\caption{{Controlled mini-batch Algorithm (\CMA)}}
\label{alg:CMA2}
\end{algorithm}

\subsubsection{\CMA\ Convergence analysis}\label{sec:CMA2_conv}

We first recall that the Linesearch procedure in Algorithm \ref{alg:EDFL2}
 is well defined; namely it terminates in a finite number of iterations. 
\begin{lemma}\label{EDFL:welldefined}
     Algorithm \ref{alg:EDFL2} is well defined, i.e. it determines in a finite number of steps a scalar $\alpha^k$ such that 
\begin{equation}\label{EDFLS:suffred}
f(w^k+\alpha^kd^k) \leq f(w^k) -\gamma\alpha^k\|d^k\|^2.
\end{equation}
\end{lemma}
\begin{proof}
First of all, we prove that Algorithm \ref{alg:EDFL2} is well-defined. In particular, we show that the while loop cannot infinitely cycle.  Indeed, if this was the case, that would imply
\begin{equation*}
    f\left(w^k +\frac{\zeta^k}{\delta^j} d^k\right) \leq f(w^k) -\gamma\left(\frac{\zeta^k}{\delta^j}\right)\|d^k\|^2  \quad \forall j=1,2,\dots
\end{equation*}
so that for $j\to \infty$ we would have $(\zeta^k/\delta^j)\to \infty$ which contradicts the compactness of ${\cal L}(w^0)$ and continuity of $f$.

{In order to prove (\ref{EDFLS:suffred}), we consider separately the two cases i) $\alpha^k = 0$ and ii) $\alpha^k > 0$. In the first case, (\ref{EDFLS:suffred}) is trivially satisfied.

In case ii),  the condition at step 4  implies
\[
  f(w^k+\zeta^kd^k) \leq f(w^k) -\gamma\zeta^k\|d^k\|^2.
\]
Then, if $\alpha^k = \zeta^k$, the condition is satisfied. Otherwise, if $\alpha^k > \zeta^k$, the stopping condition of the while loop, along with the fact that we already proved that the while loop could not infinitely cycle, is such that we have
\[
   f(w^k+\alpha^k d^k) \leq f(w^k) - \gamma\alpha^k\|d^k\|^2
\]
which again proves (\ref{EDFLS:suffred}), thus concluding the proof.
}
\end{proof}

We now prove that the sequence generated by \CMA\ remains in the level  set ${\cal L}(w^0)$ which is compact by Assumption \ref{ass:compact}, thus satisfying the assumption required in Proposition \ref{limtildew_general}.

\begin{lemma}
Let $\{w^k\}$ be the sequence of points produced by algorithm \CMA. Then, $\{w^k\}\subseteq{\cal L}(w^0)$.
\end{lemma}
\begin{proof}
By definition of algorithm \CMA\ and recalling Lemma \ref{EDFL:welldefined},  for every iteration $k$, we have that
 $f(w^{k+1})\leq f( y^k)$ (Step \ref{step20_CMA2}). Furthermore, exactly one of the following cases happens
\begin{itemize}
    \item[i)] Step \ref{cma2:step11} is executed, i.e. $f(y^k) \leq f(w^k) -\gamma\zeta^k$;
    \item[ii)] Step \ref{step18_CMA2} is executed, i.e. $f(y^k) \leq f(w^k) -\gamma\alpha^k\|d^k\|^2$;
    \item[iii)] Step \ref{cma2:step9} is executed, then, we have that 
    \begin{itemize}
        \item $f(y^k) \leq f(w^0)$ or
        \item $f(y^k)= f(w^k)$;
    \end{itemize}
    \item[iv)] Step \ref{step19_CMA2} is executed, then, we have that
    \begin{itemize}
        \item $f(y^k) = f(w^k)$ or
        \item $f(y^k) = f(\widetilde w^k) \leq f(w^0)$ or
        \item $f(y^k) = f(w^k+\alpha^k d^k) \leq f(w^k+\zeta^k d^k) =f(\widetilde w^k)\leq f(w^k) - \gamma\zeta^k\|d^k\|^2$.
    \end{itemize}
\end{itemize}
Then, for every $k$, we have in particular that either $f(w^{k+1})\leq f(w^0)$ or $f(w^{k+1})\leq f(w^k)$. Now, let us split the iteration sequence into two subsets, namely $K$ and $\bar K$, such that
\[
\begin{split}
   K & = \{k\in\mathbb{N}:k\geq 0, f(w^{k}) \leq f(w^0)\},\\ 
   \bar K & = \{k\in\mathbb{N}:k > 0, f(w^{k})\leq f(w^{k-1})\}.
\end{split}
\]
Note that, in particular, $0\in K$, i.e. $K\neq\emptyset$. It obviously results that 
\begin{equation}\label{subseqinlevel1}
\{w^{k}\}_K\subseteq{\cal L}(w^0).
\end{equation}
On the other hand, for every $k\in\bar K$, let $m_k$ be the biggest index such that $m_k < k$ and $m_k\in K$. Since $K\neq\emptyset$, $m_k$ is always well-defined. Then, by definition of $K$, $\bar K$ and $m_k$, we can write
\[
   f(w^k) \leq f(w^{k-1}) \leq \dots \leq f(w^{m_k})\leq f(w^0)
\]
Hence, when $k\in\bar K$ we still have that $f(w^k)\leq f(w^0)$, so that 
\begin{equation}\label{subseqinlevel2}
\{w^k\}_{\bar K}\subseteq{\cal L}(w^0).
\end{equation}
Then, the proof is concluded recalling (\ref{subseqinlevel1}) and (\ref{subseqinlevel2}).

\end{proof}

\noindent Now we show that the inner stepsize $\zeta^k$ goes to zero as $k\to \infty$.
\begin{proposition}\label{th:CMA2_a_0}
Let $\{\zeta^k\}$ be the sequence of steps produced by Algorithm \CMA, then
\[
 \lim_{k\to\infty}\zeta^k = 0.
\]
\end{proposition}
\begin{proof}
In every iteration, either $\zeta^{k+1} = \zeta^k$ or $\zeta^{k+1} = \theta\zeta^k < \zeta^k$. Therefore, the sequence $\{\zeta^k\}$ is monotonically non-increasing. Hence, it results
\[
\lim_{k\to\infty} \zeta^k = \bar\zeta \geq 0.
\]
Let us suppose, by contradiction, that $\bar\zeta > 0$. If this were the case, there should be an iteration index $\bar k\geq 0$ such that, for all $k\geq\bar k$, $\zeta^{k+1} = \zeta^k = \bar\zeta$. Namely, for all iterations $k\geq \bar k$, step \ref{cma2:step11} or \ref{step18_CMA2} are always executed. Then, for all $k\geq\bar k$ we have $f(w^{k+1})\leq f(w^k)$.

Let us now denote by
\begin{eqnarray*}
  K^\prime & = & \{k:\ k\geq\bar k,\ \mbox{and step \ref{cma2:step11} is executed}\}\\
  K^{\prime\prime} & = & \{k:\ k\geq\bar k,\ \mbox{and step \ref{step18_CMA2} is executed}\}  
\end{eqnarray*}

{
Let us first suppose that $K^\prime$ is infinite. Then, for $k\in K^\prime$ and $k\geq \bar k$, we would have that infinitely many times
\[
   \gamma\bar\zeta \leq f(w^k) - f(w^{k+1}).
\]
Then, taking the limit for $k\to\infty$ and $k\in K^\prime$, we get a contradiction with the compactness of the level set and the continuity of the function $f$.}

\par\smallskip

On the other hand, let us suppose that $K^{\prime\prime}$ is infinite. Then, thanks to Lemma \ref{EDFL:welldefined}, we would have that, for $k\in K^{\prime\prime}$,
\[
 f(w^{k+1}) \leq f(w^k+\alpha^kd^k) \leq {f(w^k)} -\gamma\alpha^k\|d^k\|^2.
\]
By taking the limit in the above relation and recalling the compactness of the level set ${\cal L}(w^0)$ and continuity of $f$, we have that $f(w^k)\to \bar f\in \R, k\to \infty$ and we would obtain
\[
\lim_{k\to\infty} \alpha^k\|d^k\|^2 = 0.
\]
But then, for $k\in K^{\prime\prime}$ and sufficiently large, it would happen that 
\[
 \alpha^k\|d^k\|^2 \leq \tau\bar\zeta
\]
which means that Algorithm \CMA\ would execute step \ref{cma2:step9} setting $\zeta^{k+1} = \theta\bar\zeta$ thus decreasing $\zeta^{k+1}$ below $\bar\zeta$. This contradicts our initial assumption and concludes the proof. 
\end{proof}

We are now ready to provide the main convergence result for \CMA.

\begin{proposition}
Let Assumption \ref{ass:lipschtizagradient} hold and let $\{w^k\}$ be the sequence of points produced by Algorithm \CMA. Then, $\{w^k\}$ admits limit points and (at least) one of them is stationary.
\end{proposition}
\begin{proof}
By the compactness of ${\cal L}(w^0)$ and considering that $w^k\in{\cal L}(w^0)$ for all $k$, we know that $\{w^k\}$ admits limit points.

Now, let us introduce the following set of iteration indices
\[
 K = \{k:\ \zeta^{k+1} = \theta\zeta^k\}
\]
and note that, since $\lim_{k\to\infty}\zeta^k = 0$, $K$ must be infinite.  Let $\bar w$ be any limit point of the subsequence $\{w^k\}_K$, then, by Proposition \ref{limtildew_general}, a subset $\bar K\subseteq K$ exists such that
\begin{subequations}\label{teo:cmaconv1}
\begin{eqnarray}
&& \label{subeqa}\lim_{k\to\infty, k\in \bar K} w^k = \bar w,\\
&& \label{subeqb}\lim_{k\to\infty, k\in \bar K} d^k = -\nabla f(\bar w),\\
&& \label{subeqc}\lim_{k\to\infty, k\in \bar K} \zeta^k = 0
\end{eqnarray}
\end{subequations}
Then, let us now split the set $\bar K$ into two further subsets, namely,
\begin{eqnarray*}
 K_1 & = & \{k\in \bar K:\ \|d^k\| \leq \tau\zeta^k\},\\
 K_2 & = & \{k\in \bar K\setminus K_1: \widetilde\alpha^k\|d^k\|^2 \leq \tau\zeta^k\}.
\end{eqnarray*}
Note that, $K_1$ and $K_2$ cannot be both finite.

First, let us suppose that $K_1$ is infinite. Then, the definition of $K_1$, (\ref{subeqb}) and (\ref{subeqc}) imply that
\[
 \lim_{k\to\infty,k\in K_1} \|d^k\| = \|\nabla f(\bar w)\| = 0
\] 
thus concluding the proof in this case.
\par\smallskip

Now, let us suppose that $K_2$ is infinite. In this case, we proceed by contradiction and assume that $\bar w$ is not stationary, i.e. $\|\nabla f(\bar w)\| > 0$. By the definition of $K_2$ and (\ref{subeqc}), we obtain
\[
\lim_{k\to\infty,k\in K_2} \widetilde\alpha^k\|d^k\|^2 = 0,
\]
which, recalling (\ref{subeqb}) and the fact that $\|\nabla f(\bar w)\| > 0$, yields
\begin{equation}\label{eq:alfatildezero}
\lim_{k\to\infty,k\in K_2} \widetilde\alpha^k = 0.
\end{equation}
Now, let us further partition $K_2$ into two subsets, namely
\begin{itemize}
    \item[i)] $\bar K_2 = \{k\in K_2:\ \widetilde\alpha^k = 0\}$
    \item[ii)] $\widehat K_2 = \{k\in K_2:\ \widetilde\alpha^k > 0\}$.
\end{itemize}

Let us first suppose that $\bar K_2$ is infinite. This means that, by the definition of algorithm EDFL, 
\[
 \frac{f(w^k) - f(w^k+\zeta^kd^k)}{\zeta^k} < \gamma \|d^k\|^2.
\]
Then, by the Mean-Value Theorem, a number $\hat\alpha^k\in (0,\zeta^k)$ exists such that
\[
 -\nabla f(w^k + \hat\alpha^kd^k)^Td^k < \gamma \|d^k\|^2.
\]
Taking the limit in the above relation, considering (\ref{subeqa}), (\ref{subeqb}) and (\ref{subeqc}), we obtain
\[
 \|\nabla f(\bar w)\|^2 \leq \gamma \|\nabla f(\bar w)\|^2
\]
which, recalling that $\gamma \in (0,1)$, gives
\[
\|\nabla f(\bar w)\| = 0,
\]
which contradicts $\|\nabla f(\bar w)\|>0$.
Now, let us suppose that $\widehat K_2$ is infinite. This means that 
\[
 \frac{f(w^k) - f(w^k+(\widetilde\alpha^k/\delta)d^k)}{\widetilde\alpha^k/\delta} < \gamma \|d^k\|^2.
\]
Then, reasoning as in the previous case, and considering (\ref{eq:alfatildezero}), we again can conclude that $\|\nabla f(\bar w)\| = 0$, which is again a contradiction and concludes the proof.
\end{proof}

\subsection{Nonmonotone Controlled mini-batch Algorithm}
In this section, we propose a nonmonotone version \NMCMA\  of the Controlled mini-batch Algorithm.

This  algorithm is motivated by relaxing the acceptance criteria so to accept even more frequently the trial point $\widetilde w^k$ thus possibly decreasing the need for function evaluations and increasing the computational efficiency.

In particular, we introduce a nonmonotone acceptance criterion both in Step \ref{cma:step10} of \CMA\ and in steps \ref{edfl:step4} and \ref{edfl:step7} of Algorithm \ref{alg:EDFL2}. 
Following classical approaches in nonmonotone scheme \cite{grippo1986nonmonotone,grippo2007}, in these conditions we substitute  the value $f(w^k)$ with a reference value $R^k$ which does not exceed the value of the objective function in the last $M$ iterations, being $M$ the nonmonotone memory.

Following \cite{grippo2007}, we introduce the Nonmonotone Derivative Free Linesearch procedure NMEDFL in Algorithm \ref{alg:NMEDFL2}. We note that the use of the reference value $R^k$ is not the only difference w.r.t. the monotone version. Indeed, Steps \ref{NMEDFL:STEP4} and  \ref{NMEDFL:STEP7} require a parabolic reduction condition on the stepsize $\alpha$ instead of a linear one.

\begin{algorithm}[htb]
\caption{Nonmonotone Extrapolation DF Linesearch (NMEDFL)}
\label{alg:NMEDFL2}
\begin{algorithmic}[1]
    \State Input $({R^k},w^k,d^k,\zeta^k;\gamma,\delta)$: $w^k\in \R^n,d^k\in \R^n, R^k\in \R, \zeta^k>0$;  $\gamma\in(0,1), \delta\in(0,1)$
%    \State Data  $\delta\in(0,1)$ change Laura 23 october 2022
    \State Set $\alpha=\zeta^k$
    \If{$f(w^k +\alpha d^k)> {R^k}-\gamma{(\alpha)^2}\|d^k\|^2$}
      \label{NMEDFL:STEP4}  \State Set $\alpha^k = 0$ and \textbf{return} % $\alpha^k$
    \EndIf
    \While{$f(w^k +(\alpha/\delta) d^k)\leq \label{NMEDFL:STEP7}\min\{R^k-\gamma{(\alpha/\delta)^2}\|d^k\|^2, f(w^k+\alpha d^k)\}$}
        \State Set $\alpha=\alpha/\delta$
    \EndWhile
    \State Set $\alpha^k = \alpha$ and \textbf{return}
    \State Output:  $\alpha^k$
 \end{algorithmic}
 \end{algorithm}

The scheme of  the overall nonmonotone algorithm \NMCMA\ is provided in Algorithm \ref{alg:NMCMA}. We note that Step \ref{nmcma:step6}  of \NMCMA\  relaxes the monotonic requirement of Step \ref{cma:step10}  of \CMA\  but strengthens  the sufficient reduction condition, being $\max\{\zeta^k,\zeta^k\|d^k\|\}\ge \zeta^k$.
 Furthermore, we have no more the additional freedom of selecting $w^{k+1}$ such that $f(w^{k+1})\le f(w^k + \alpha^k d^k) $. This is due to the need of controlling the distance $\|w^{k+1}-w^k\|$ among iterates. Thus, it is not possible to state that \NMCMA\ is computationally less expensive than \CMA\ .

\begin{algorithm}[htb]
\begin{algorithmic}[1]
    \State Set $ \zeta^0 >0, \theta\in(0,1), \tau > 0, M\in \N$, $\gamma\in(0,1), \delta\in(0,1)$      % change 23 ottobre 2022 Laura 

    \State Let $w^0\in \R^n,k=0$
    %{\While {(stopping criterion not met)}}
    % change 23 febbraio 2021 Laura 6 Giampo
    \For {$k=0,1,2\dots $}
        \State Set $\displaystyle R^k = \max_{0\leq j\leq\min\{k,M\}}\{f(w^{k-j})\}$.
    \State Compute $(\widetilde w^k, d^k)$ = \texttt{Inner\_Cycle}($w^k,\zeta^k$)
    \If{$f(\widetilde w^k) \leq {R^k} - \gamma{\max\{}\zeta^k{,\zeta^k\|d^k\|\}}$} %- \gamma(\zeta^k)$}
    \label{nmcma:step6}\State  ${\alpha^k=\zeta^k}, \ \zeta^{k+1} = \zeta^k$\label{nmcma:step11} %Set $y^k = \widetilde w^k$
    \Else
     \If{$\|d^k\| \leq \tau\zeta^k$}
        \State $\zeta^{k+1}=\theta \zeta^k$ , $\alpha^k = 0$ \label{step10_NMCMA}
    \Else
         \State \label{nmcma:step23} $\alpha^k =NMEDFL(R^k,d^k,w^k,\zeta^k;\gamma,\delta)$  %change Laura 23 ottobre 2022
         \If {${(\alpha^k)^2}\|d^k\|^2 \leq\tau\zeta^k$}
            \State Set $\zeta^{k+1}=\theta \zeta^k$  \label{step19_NMCMA}
         \Else
            \State Set $\zeta^{k+1}=\zeta^k$ \label{step18_NMCMA}
         \EndIf
    \EndIf
    \EndIf
    
    \State { $w^{k+1} =  w^k + \alpha^k d^k$}% {\color{red}Set $w^{k+1} = y^k$}
   \EndFor
\end{algorithmic}
\caption{Nonmonotone Controlled mini-batch Algorithm (\NMCMA)}
\label{alg:NMCMA}
\end{algorithm}

\subsubsection{\NMCMA\ Convergence analysis}\label{sec:NMCMA_conv}
We first recall that the Linesearch procedure in Algorithm \ref{alg:NMCMA} is well defined, namely it terminates in a finite number of iterations. 
\begin{lemma}\label{lemma:NMEDFL:welldefined}
     Algorithm \ref{alg:NMCMA} is well defined, {
     i.e. it determines in a finite number of steps a scalar $\alpha^k$ satisfying 
\begin{equation}\label{NDFLS2:suffred}
f(w^k+\alpha^kd^k) \leq R^k -\gamma(\alpha^k)^2\|d^k\|^2,
\end{equation}
where $R^k$ is a reference value such that
\[
f(w^k) \le R^k \le \max_{0\leq j\leq\min\{k,M\}}\{f(w^{k-j})\}.
\]                            
 }
\end{lemma}
\begin{proof}
Algorithm NMEDFL cannot cycle indefinitely within the while cycle because that would imply

\begin{equation*}
    f\left(w^k +\frac{\zeta^k}{\delta^j} d^k\right) \leq R^k
    -\gamma\left(\frac{\zeta^k}{\delta^j}\right)^2\|d^k\|^2  \quad \forall j=1,2,\dots
\end{equation*}
so that for $j\to \infty$ we would have $(\zeta^k/\delta^j)\to \infty$ which contradicts the compactness of ${\cal L}_0$ and continuity of $f$.

{In order to prove (\ref{NDFLS2:suffred}), we consider separately the two cases i) $\alpha^k = 0$ and ii) $\alpha^k > 0$. In the first case, (\ref{NDFLS2:suffred}) is trivially satisfied given the definition of $R^k$.

In case ii), since the if condition at step 4 is not satisfied, we have that 
\[
  f(w^k+\zeta^kd^k) \leq R^k -\gamma(\zeta^k)^2\|d^k\|^2.
\]
Then, if the final $\alpha^k$ is equal to $\zeta^k$, the condition is satisfied. Otherwise, if $\alpha^k > \zeta^k$, the conditions of the while loop are such that we have

\[
   f(w^k+\alpha^kd^k) \leq R^k - \gamma(\alpha^k)^2\|d^k\|^2
\]
thus the proof is completed.
}
\end{proof}

{
\begin{lemma}\label{lemma:nonmondecr}
Let $\{w^k\}$ and $\{R^k\}$ be the sequences produced by algorithm \NMCMA. Then, for every iteration $k$ we have
\[
  f(w^{k+1}) \leq R^k -\sigma(\|w^{k+1}-w^k\|).
\]
where $\sigma:\Re^+\to\Re^+$ is a forcing function, i.e. a function such that, for any sequence $\{t^k\}\subseteq \Re^+$, 
$\lim_{k\to\infty}\sigma(t^k) = 0$ implies $\lim_{k\to\infty}t^k = 0$. 
\end{lemma}
\begin{proof}
We note that, by definition of algorithm \NMCMA, at every iteration, 
\begin{itemize}
    \item[i)] either Step \ref{nmcma:step11} is executed, in which case we have that
\[
   f(w^{k+1}) = f(w^k+\zeta^kd^k) \leq R^k -\gamma \zeta^k\|d^k\| = R^k - \gamma \|w^{k+1}-w^k\|
\]
\item[ii)] or Step \ref{nmcma:step23} is executed in which case, recalling Lemma  \ref{lemma:NMEDFL:welldefined}, we have 
\[
   f(w^{k+1}) = f(w^k + \alpha^kd^k)\leq R^k -\gamma (\alpha^k)^2\|d^k\|^2 = R^k -\gamma\|w^{k+1}-w^k\|^2.
\]
\end{itemize}
Then, we have 
\[
f(w^{k+1}) \leq R^k - \gamma\min\{\|w^{k+1}-w^k\|,\|w^{k+1}-w^k\|^2\},
\]
which concludes the proof.
\end{proof}

{We now prove that the sequence produced by Algorithm \ref{alg:NMCMA} remains in the compact level set ${\cal L}(w^0)$ and that the distance among successive iterates is driven to zero.
}

\begin{proposition}\label{propgrippo2007_1}
Let $\{w^k\}$ be the sequence of points produced by algorithm \NMCMA. 

Then,
\begin{itemize}
    \item [i)] $w^k\in{\cal L}(w^0)$ for all $k$;
    \item [ii)] $\{R^k\}$ and  $\{f(w^k)\}$ are both convergent and have the same limit $R^*$;
    \item[iii)] $\lim_{k\to\infty}\|w^{k+1}-w^k\| = 0$;

\end{itemize}

\end{proposition}
\begin{proof}
By Assumption \ref{ass:lipschtizagradient} the function $f$ is Lipschitz continuous and by Assumption \ref{ass:compact}, $f$ is bounded below. 
Recalling Lemma \ref{lemma:nonmondecr}, the fact that $f$ is continuously differentiable and that ${\cal L}(w^0)$ is compact, we have that all the assumptions of Lemma 1 in \cite{grippo2007} are satisfied. For the sake of completeness, the statement and proof of this Lemma are reported in 
Lemma \ref{lemma1[19]} in the Appendix). Hence, the proof exactly follows from \cite[Lemma 1]{grippo2007} except that we have still to prove that
\[
\lim_{k\to\infty}R^k = \lim_{k\to\infty}f(w^k).
\]
This is proved by noting that $\{f(w^k)\}$
converges to a limit, and by recalling the definition of $R^k$, i.e. 
$ R^k = \max_{0\leq j\leq\min\{k,M\}}\{f(w^{k-j})\}$.

\end{proof}

}

\noindent We now show that the inner stepsize $\zeta^k$ goes to zero as $k\to \infty$.

\begin{proposition}\label{th:NMCMA_a_0}
Let $\{\zeta^k\}$ be the sequence of steps produced by Algorithm \NMCMA, then
\[
 \lim_{k\to\infty}\zeta^k = 0.
\]
\end{proposition}
\begin{proof}
The first part of the  proof is similar to the one of Proposition \ref{th:CMA2_a_0}. In every iteration, either $\zeta^{k+1} = \zeta^k$ or $\zeta^{k+1} = \theta\zeta^k < \zeta^k$. Therefore, the sequence $\{\zeta^k\}$ is monotonically non-increasing. Hence, it results
\[
\lim_{k\to\infty} \zeta^k = \bar\zeta \geq 0.
\]
Let us suppose, by contradiction, that $\bar\zeta > 0$. If this were the case, there should be an iteration index $\bar k\geq 0$ such that, for all $k\geq\bar k$, $\zeta^{k+1} = \zeta^k = \bar\zeta$. Namely, for all iterations $k\geq \bar k$, step \ref{step18_NMCMA} or \ref{cma2:step11} are always executed. 

Let us now denote by
\begin{eqnarray*}
  K^\prime & = & \{k:\ k\geq\bar k,\ \mbox{and step \ref{nmcma:step11} is executed}\},\\
  K^{\prime\prime} & = & \{k:\ k\geq\bar k,\ \mbox{and step \ref{step18_NMCMA} is executed}\}.  
\end{eqnarray*}

{
Let us first suppose that $K^\prime$ is infinite. Then, for $k\in K^\prime$ and $k\geq \bar k$, we would have that infinitely many times
\[
   \gamma\bar\zeta \leq R^k - f(w^{k+1}).
\]
Then, taking the limit for $k\to\infty$ and $k\in K^\prime$, and recalling the results of Proposition \ref{propgrippo2007_1}, we get a contradiction with $\bar\zeta > 0$.}

\par\smallskip

On the other hand, let us suppose that $K^{\prime\prime}$ is infinite. Then, infinitely many times, we would have that, for $k\in K^{\prime\prime}$,
\[
 f(w^{k+1}) = f(w^k+\alpha^kd^k) \leq {R^k} -\gamma(\alpha^k)^2\|d^k\|^2.
\]
By taking the limit in the above relation and recalling the compactness of the level set ${\cal L}(w^0)$ and continuity of $f$, we would obtain
\[
\lim_{k\to\infty} (\alpha^k)^2\|d^k\|^2 = 0.
\]
But then, for $k\in K^{\prime\prime}$ and sufficiently large, it would happen that 
\[
 (\alpha^k)^2\|d^k\|^2 \leq \tau\bar\zeta
\]
which means that Algorithm \NMCMA\ would have executed step \ref{step19_NMCMA} setting $\zeta^{k+1} = \theta\bar\zeta$, thus reducing $\zeta^{k+1}$ below $\bar\zeta$. This contradicts our initial assumption and concludes the proof. 
\end{proof}

We are now ready to provide the convergence result for \NMCMA.

\begin{proposition}
Let Assumption \ref{ass:lipschtizagradient} hold and
let $\{w^k\}$ be the sequence of points produced by Algorithm \NMCMA. Then, $\{w^k\}$ admits limit points and (at least) one of them is stationary.
\end{proposition}
\begin{proof}

By the compactness of ${\cal L}(w^0)$ and considering that $w^k\in{\cal L}(w^0)$ for all $k$, we know that $\{w^k\}$ admits limit points.

Now, let us introduce the following set of iteration indices
\[
 K = \{k:\ \zeta^{k+1} = \theta\zeta^k\}
\]
and note that, since $\lim_{k\to\infty}\zeta^k = 0$, $K$ must be infinite.  Let $\bar w$ be any limit point of the subsequence $\{w^k\}_K$, then, by Proposition \ref{limtildew_general}, a subset $\bar K\subseteq K$ exists such that
\begin{subequations}\label{teo:nmcmaconv1}
\begin{eqnarray}
&& \label{subeqanm}\lim_{k\to\infty, k\in \bar K} w^k = \bar w,\\
&& \label{subeqbnm}\lim_{k\to\infty, k\in \bar K} d^k = -\nabla f(\bar w),\\
&& \label{subeqcnm}\lim_{k\to\infty, k\in \bar K} \zeta^k = 0
\end{eqnarray}
\end{subequations}

Then, let us now split the set $\bar K$ into two subsets, namely,
\begin{eqnarray*}
 K_1 & = & \{k\in \bar K:\ \|d^k\| \leq \tau\zeta^k\},\\
 K_2 & = & \{k\in \bar K\setminus K_1: (\alpha^k)^2\|d^k\|^2 \leq \tau\zeta^k\}.
\end{eqnarray*}
Note that, $K_1$ and $K_2$ cannot be both finite.

First, let us suppose that $K_1$ is infinite. In this case, (\ref{teo:nmcmaconv1}) imply that
\[
 \lim_{k\to\infty,k\in K_1} \|d^k\| = \|\nabla f(\bar w)\| = 0
\] 
thus concluding the proof in this case.

Then, let us suppose that $K_2$ is infinite. In this case, we proceed by contradiction and assume that $\bar w$ is not stationary, i.e. $\|\nabla f(\bar w)\| > 0$. By the definition of $K_2$ and considering (\ref{subeqcnm}), we obtain
\[
\lim_{k\to\infty,k\in K_2} (\alpha^k)^2\|d^k\|^2 = 0,
\]
which, considering (\ref{subeqbnm}), yields

\begin{equation}\label{eq:tildealfa2=0}
\lim_{k\to\infty,k\in K_2} \alpha^k = 0.
\end{equation}

Now, let us further partition $K_2$ into
\begin{itemize}
    \item[i)] $\bar K_2 = \{k\in K_2:\ \alpha^k = 0\}$
    \item[ii)] $\widehat K_2 = \{k\in K_2:\ \alpha^k > 0\}$.
\end{itemize}

Let us first suppose that infinitely many times point (i) happens. This means that 
\[
 \frac{f(w^k) - f(w^k+\zeta^kd^k)}{\zeta^k} < \gamma \zeta^k\|d^k\|^2.
\]
Then, by the Mean-Value Theorem, a number $\hat\alpha^k\in (0,\zeta^k)$ exists such that
\[
 -\nabla f(w^k + \hat\alpha^kd^k)^Td^k < \gamma \zeta^k\|d^k\|^2.
\]
Taking the limit in the above relation for $k\to\infty$, $k\in\bar K_2$, recalling (\ref{teo:nmcmaconv1}), we obtain
\[
 \|\nabla f(\bar w)\|^2 \leq 0
\]
i.e. $\|\nabla f(\bar w)\| = 0$.

Now, let us suppose that point (ii) happens. This means that 
\[
 \frac{f(w^k) - f(w^k+(\alpha^k/\delta)d^k)}{\alpha^k/\delta} < \gamma (\alpha^k/\delta)\|d^k\|^2.
\]

Then, reasoning as in the previous case, and recalling \eqref{eq:tildealfa2=0} we again can conclude that $\|\nabla f(\bar w)\| = 0$, thus concluding the proof.
\end{proof}

\section{Numerical results}\label{sec:CMA_numerical}

In this section, we report  numerical results when comparing the implementations of  \CMA\ and \NMCMA\ with   full batch and mini-batch methods.
We consider problems with different values of $n$ and samples $P$ to check how the algorithmic performance changes varying the size of the optimization
problem and the size of the batch.

The aim of the computational experience is twofold: on the one hand, we aim to analyze the additional computational effort required by the  \CMA\ and \NMCMA\  with respect to a standard online method and the effectiveness of the  updating rules for the stepsize to avoid it to be driven to zero too fast; on the other hand, the quality of the solution obtained in a limited time in comparison with online methods and classical batch methods. 

\subsection{Experimental testbeds}
The performances of the algorithms were tested on the minimization problem raised when minimizing the $\ell_2$ regularized least square training error of a Deep  Neural Network (DNN),  
namely
$$\underset{w}{\text{min}}f(w)=\frac{1}{P}\sum_{p=1}^P\|\tilde{y}(w;x_p)-y_p\|^2 +\rho\Vert w\Vert ^2
$$
where $\lbrace x_p,y_p\rbrace_{p=1}^P$, with $x_p \in {\rm I\!R}^d$ and $y_p \in {\rm I\!R}^m$ are the training data and  $\tilde{y}_p$ is the output of the DNN with weights $w\in\R^n$.
The regularization parameter $\rho$ is set equal to $10^{-6}$. 

The DNN is organized in $L$ hidden layers.
Each neuron  is characterized by an activation function $g(\cdot)$, that we set to be the sigmoidal  function for all the neurons,
with the only exception of the output layer which has a linear activation function. 
Denoting with $w_\ell$ the weights  from layer $\ell$ to layer $\ell+1$,
the output takes the form
$$\tilde{y}(w;x_p)=w_{L}g(w_{L-1}g(w_{L-2}\dots g(w_1^T x_p))).$$
We considered four different neural network architectures classified into small and large sizes as reported in Table \ref{tab:DNN_arch}.
For the sake of simplicity, we considered only networks with the same number of neurons per hidden layer, and we denote as
$[L\times N]$ the  neural network with $L$ hidden layers and $N$ neurons per layer.
The dimension $n$ of the weights  $w$  depends on the number of neurons per layer $N^j$, $j=1,\dots,L$ and on the input dimension $d$, and in our case is equal to $n=N(d+1)+N^2(L-1)$.

\begin{table}[]
    \centering
    \caption{Deep Network architectures}
    \begin{tabular}{c|cccc|ccc}
&\multicolumn{4}{c|}{small networks} & \multicolumn{3}{c}{large networks} \\
&          $[1\times 50]$ & $[3\times 20]$ & $[5\times 50]$ & $[10\times 50]$ &$[12\times 300]$ & $[12\times 500]$ & $[30\times 500]$ \\[.4em] \hline
         $L$ & 1 &3&5&10&12&12&30\\ 
         $N$ &50&20&50&50&300&500&500\\
%         $n$ & $50 d +50$ & $20(d+1)$\\
    \end{tabular}
    \label{tab:DNN_arch}
\end{table}

All the algorithms were tested over a set of eleven problems with a different number of samples $P$ and features $d$ taken from open libraries \cite{keel,UCI,sido}. The datasets have been divided into two groups: \textit{small-size} datasets, where the number of samples and/or features is small enough to   be  tackled by standard batch methods; \textit{big-size} datasets, where the number of samples and/or features challenges fully batch methods and motivates the use of mini-batch algorithms. 

Table \ref{tab:CMA_dataset} provides a description of the datasets  by reporting the number of samples in the training  set and the number of features, along with the number of variables of each optimization problem corresponding to different neural architectures $[L\times N]$. Datasets are sorted according to their dimensions, with small problems reported first, and large problems reported later. The order within each group is alphabetical. All the datasets are publicly available, and the source where they can be found is reported as a reference for each dataset.

\begin{table}
  \centering
    \begin{adjustbox}{width=\textwidth}
        \begin{tabular}{l|l|rrr|rrrr|rrr}
            \hline
            &  & &  &  & \multicolumn{7}{c}{$n$ in K } \\ 
            &   &&&&\multicolumn{4}{c|}{small network} & \multicolumn{3}{c}{large network} \\
            &Dataset & \multicolumn{1}{r}{$P$} &  & \multicolumn{1}{r|}{$d$} & \multicolumn{1}{l}{$[1\times 50]$} & \multicolumn{1}{l}{$[3\times 20]$} & \multicolumn{1}{l}{$[5\times 50]$} & \multicolumn{1}{l}{$[10\times 50]$} & \multicolumn{1}{|l}{$[12\times 300]$} & \multicolumn{1}{l}{$[12\times 500]$} & \multicolumn{1}{l}{$[30\times 500]$}  \bigstrut[b]\\
            \hline
            \parbox[t]{2mm}{\multirow{5}{*}{\rotatebox[origin=c]{90}{small}}} 
            &Ailerons \cite{keel} & 10312 &   & 41    & 2.1  & 1.64  & 12.1 & 24.6 & 935 & 2557& 7567 \bigstrut[t]\\
            & Bejing Pm25 \cite{UCI} & 31317 &  & 48    & 2.45  & 1.78  & 12.5 & 25 & 935 & 2558 & 7568 \\
            & Bikes Sharing \cite{UCI} & 13034 &   & 59    & 3.0  & 2.0  & 13 & 2.55 & 936 & 2559 & 7569\\
            & California \cite{keel} & 15480 &   & 9     & 0.5   & 1.0  & 10.5 & 23 & 933 & 2556 & 7556 \\
            & Mv \cite{keel}    & 30576 &  & 13    & 0.7   & 1.08  & 10.7 & 23.2 & 934 & 2556 & 7566 \bigstrut[b]\\
            \hline
            \parbox[t]{2mm}{\multirow{6}{*}{\rotatebox[origin=c]{90}{big}}}
            & BlogFeedback \cite{UCI} & 39297 &  & 281   & 14.1 & 6.44  & 24.1 & 36.6 & 948 & 2570 & 7580 \bigstrut[t]\\
            & Covtype \cite{UCI} & 435759 &  & 55    & 2.8  & 1.92  & 12.8 & 25.3 & 936 & 2558 & 7568 \\
            &Protein \cite{UCI}& 109313 &  & 75    & 3.8  & 2.32 & 13.8 & 26.3 & 937 & 2560 & 7570  \\
            &Sido \cite{sido}   & 9508  &   & 4933  & 247 & 99.5 & 257 & 269 & 1180 & 2802 & 7812 \\
            &Skin NonSkin \cite{UCI} & 183792 &  & 4     & 0.25   & 0.9   & 10.3 & 22.8 & 933 & 2555& 7566 \\
            & YearPredictionMSD \cite{UCI} & 386508 &  & 91    & 4.6  & 2.64  & 14.6 & 27.1 & 937 & 2561 & 7571 \bigstrut[b]\\
            \hline
    \end{tabular}%
    \end{adjustbox}
\caption{Dataset description  ($P=$ number of training samples; $d$= number of input features) and $n=$number of variables (in K) of the optimization problems corresponding to the network architecture $[L\times N]$}
  \label{tab:CMA_dataset}%
\end{table}%

\subsection{Algorithms settings}
Both the neural network structures and the optimization algorithms were implemented in Python version 3.8, leveraging  \textit{Numpy} (version 1.20.2) \cite{numpy} as a scientific computing library. 
Numerical tests are carried out on  an Intel(R) Core(TM) i7-3630QM CPU 2.4 GHz.

We compared the performance of \CMA\ and \NMCMA\  against a full batch method and an incremental gradient method. All the algorithms run with a maximum computational time of 100 seconds.

Since the problems are highly non-convex, numerical results might be affected by the starting point used to initialize the algorithms. Thus, a \textit{multi-start} strategy was implemented so that for each problem we compared each algorithm over five runs, each starting with a random point $w^0$ which is the same for the four algorithms.  

As a batch method, we considered  a limited-memory Quasi-Newton optimization algorithm L-BFGS \cite{NoceWrig06} that can be regarded as one of the state-of-the-art batch algorithms. We used L-BFGS from the standard implementation available in  \textit{Scipy} (version 1.6.2) \cite{scipy}. Default  hyperparameters have been used as they resulted in the best performance.

As a mini-batch method, we implemented an Incremental Gradient (IG) method, $I^k=\{1,\dots P\}$ for all $k$.
As updating rule for the stepsize $\zeta^k$ in Algorithm \ref{alg:IG}, we selected the one suggested in \cite{Goodfellow-et-al-2016}
\begin{equation}\label{eq:stepsizeIG}
    \zeta^{k+1}=\zeta^k(1-\epsilon \zeta^k)
\end{equation}
with $\zeta^0=0.5$ and $\epsilon=10^{-3}$, that gave the best performance{ out of an hyperparameter tuning phase}. 

For the sake of fairness, the same IG implementation is used for the \texttt{Inner\_Cycle} Algorithm \ref{alg:innercycle} in both \CMA\ and \NMCMA\  methods, as we aim to analyze  to what extent controlling the stepsize  in \CMA\ and \NMCMA\  affects the optimization performance. In \CMA\ and \NMCMA\ the stepsize $\zeta^k$ is automatically updated  according to the controlling steps in Algorithms \ref{alg:CMA2} and \ref{alg:NMCMA}.

The tuning of the algorithmic hyperparameters of \CMA \ and \NMCMA\ has been performed by a grid search procedure. The resulting best values are reported in Table \ref{tab:hypersetting}.

\begin{table}[h]
    \centering
    \caption{Hyperparameters settings in \CMA\ and \NMCMA\ }
    \begin{tabular}{c|cccccc}
       &  $\zeta^0$ & $  \tau$ & $\theta $ & $\gamma $ & $
     \delta $ &$M$  \\\hline
         \CMA\ & 0.5 &$10^{-2}$ &0.5 & $10^{-6}$& 0.5& - \\
         \NMCMA\ & 0.5 & $10^{-2}$ & 0.5&  $10^{-6}$ & 0.5& 5\\ 
    \end{tabular}
    \label{tab:hypersetting}
\end{table}

Further, for the sake of simplicity, in \CMA\ , Step \ref{step20_CMA2} of Algorithm \ref{alg:CMA2},  we have set $w^{k+1}=y^k=w^k+\alpha^kd^k$ .

\subsection{Performance analysis}

In this section, we analyze the computational behavior of the   \CMA\ and \NMCMA\ with the aim of 
\begin{itemize}
    \item assessing the additional effort required over the basic IG method and  the effectiveness of the  updating rules for the stepsize $\zeta^k$;
    \item comparing performance with the basic IG and the L-BFGS  methods when a restriction on the computational time is given (100 seconds). 
\end{itemize}

To this aim, we consider the set of 77 test problems ${\cal P}$, obtained by considering the 7 network architectures described in Table \ref{tab:CMA_dataset} and the 11 different dataset from Table \ref{tab:CMA_dataset}.
Each of these problems has been solved by each of the four algorithms starting from five different random points with a total of 385 runs for each algorithm.

We first observe that both  \CMA\ and \NMCMA\ require at least a function evaluation in the trial point $f(\tilde w^k)$ at the end of each epoch. 
The smaller the number of function evaluations, the less the linesearch procedure is used. 
However, this may also result in failures in accepting the trial point $\tilde w^k$ produced by the \texttt{Inner\_Cycle}  algorithm, meaning that all the computations of the \texttt{Inner\_Cycle} are lost, and the stepsize $\zeta^k$ is reduced.
To analyze the occurrence of such cases, we construct the boxplots in Figure \ref{fig:boxplot_FF} where all the runs of \CMA\ and \NMCMA\  have been considered.
In \ref{fig:boxplot_ff1}  the {average number} of function evaluations per iteration is reported for small (upper graph) and large (lower graphs) problems; if it is greater than 1.0, the LS is activated at least once. In \ref{fig:boxplot_ff2} we report the number of times that the trial point $\tilde w^k$ has been accepted. 
%\todo[inline]{esclusa o inclusa la linesearch ?}\todo{Inclusa!}

\begin{figure}[h]
    \centering     %%% not \center
        \subfigure[Average number of function evaluations per iteration]
            {\label{fig:boxplot_ff1}
            \includegraphics[width=\textwidth]{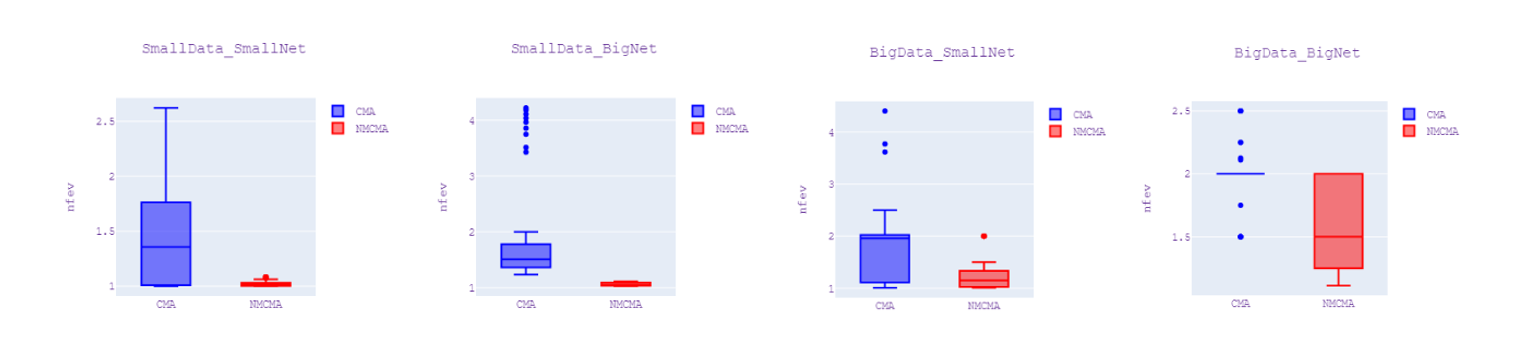}}
        \subfigure[Acceptance of the trial point $\tilde w^k$]
            {\label{fig:boxplot_ff2}
            \includegraphics[width=\textwidth]{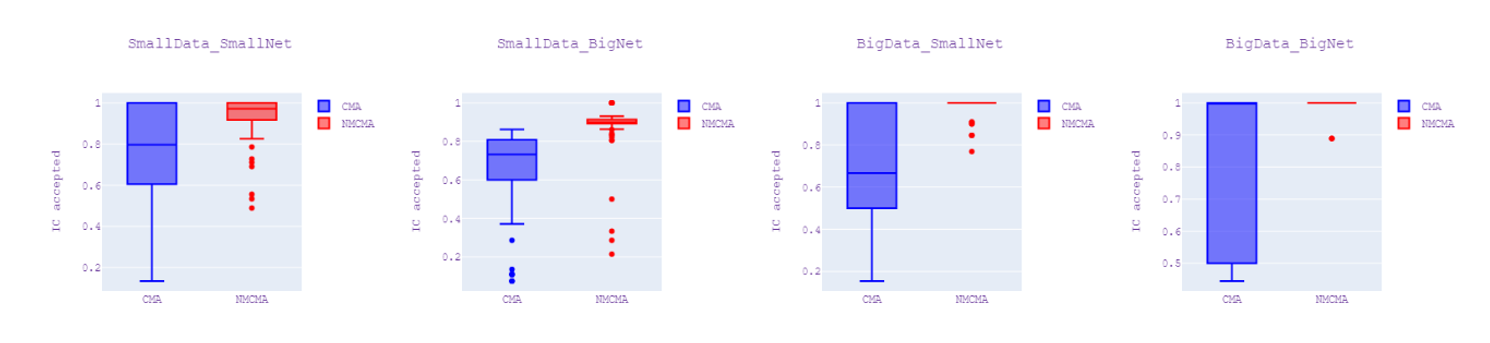}}
     \caption{Boxplot on all the runs: comparison of \CMA\  and \NMCMA\ on the average number of function evaluations and acceptance rate of  $\tilde w^k$}
    \label{fig:boxplot_FF}
\end{figure}

The boxplots of Figures \ref{fig:boxplot_ff1} and \ref{fig:boxplot_ff2} are coherent. 
Indeed, from Figure \ref{fig:boxplot_ff1}, we observe that \NMCMA\ almost always performs a single function evaluation per iteration, which is less than the iterations performed by \CMA. 
This corresponds to the more often acceptance of the trial point, namely the standard iteration of the IG method.

\begin{figure}[h]
    \centering     %%% not \center
        \subfigure[Value of the final stepsize $\zeta$:  comparison of \CMA, \NMCMA\ and IG]
            {\label{fig:boxplot_3}
            \includegraphics[width=\textwidth]{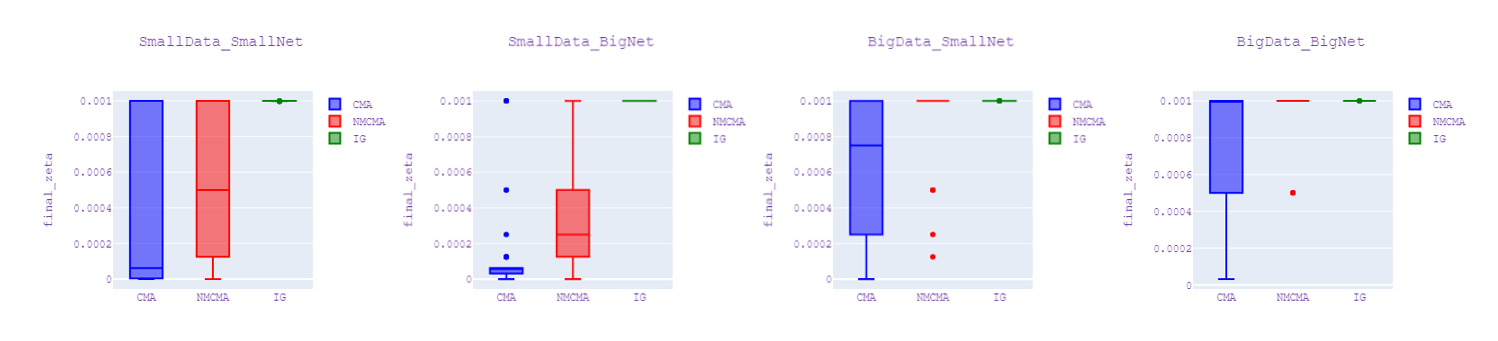}}
        \subfigure[Linesearch failures: $\alpha^k=0$]
            {\label{fig:boxplot_4}
            \includegraphics[width=\textwidth]{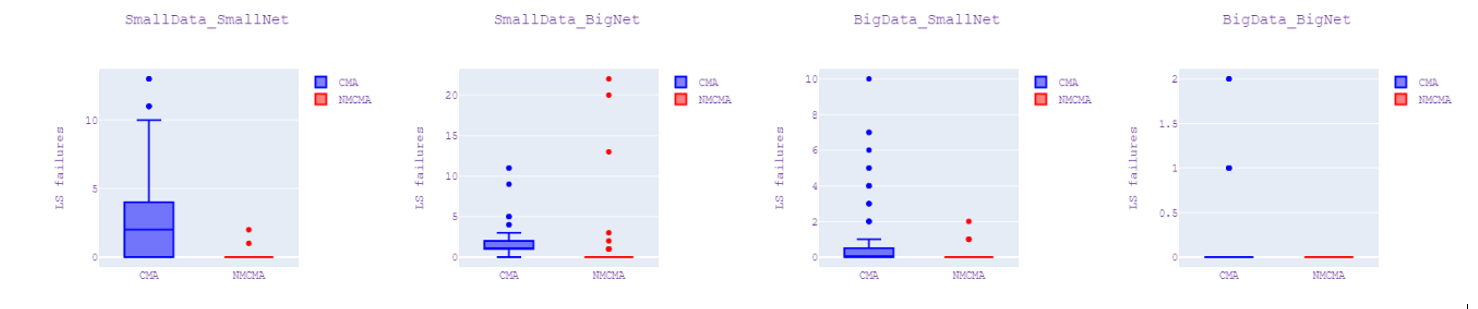}}
     \caption{Boxplot on all the runs: stepsize values and restarts}
    \label{fig:boxplot}
\end{figure}

In Figure \ref{fig:boxplot_3}, we report the average value of the final stepsize for the algorithms \CMA, \NMCMA , and IG. In general \NMCMA\ tends to maintain a larger final stepsize than both \CMA \ and IG, {which is coherent with the fact that in \NMCMA\ the trial point is  accepted more often than in \CMA}.  
The three algorithms start with the same initial value of the stepsize $\zeta^0$, however, in IG the stepsize is updated following \eqref{eq:stepsizeIG} which allows a slower decrease rather than the simple halving rule of both \CMA\ and \NMCMA. 
This may explain why the {average} final stepsize of IG is larger than in \CMA\ and \NMCMA. 
Moreover, it is interesting to underline how the range of variability of the final stepsize is bigger for both \CMA \ and \NMCMA \, suggesting that the two methods are able to tune the stepsize according to the specified instance under consideration.

Figure \ref{fig:boxplot_4} reports the number of failures of the LS in \CMA \ and \NMCMA, namely the times that $\alpha^k=0$ and a restart is needed from the last point $w^k$ so that all the computations of the \texttt{Inner\_Cycle} are lost. We note that both in  \CMA\ and \NMCMA,  $\alpha^k$ can be set to zero not only in the LS. However, we aim to visualize how the non-monotone rule helps in accepting the stepsize.

From the boxplots above, it seems that \CMA\ and \NMCMA\ accept often the trial point of the basic IG iteration, controlling the stepsize reduction, at the cost of extra calculus of objective functions that are not needed in the basic IG method.

We now aim to analyze the performance in terms of the quality of the returned solution (value of the objective function) when the computational time is restricted. In this last case, we compare also with the L-BFGS which in general decreases faster, being a full batch method.

To assess the numerical performance of the algorithms, we considered the performance profile curves \cite{dolan2002benchmarking},
 with the convergence test proposed in \cite{more2009benchmarking} when evaluating the performance of derivative-free solvers with a restricted computational budget.
In particular, given a set of solvers ${\cal S}$ and a set of test problems ${\cal P}$, a problem $p$ is \emph{solved} by $s$ if it returns a point $w$ such that
\begin{equation} \label{eq:def_solved}
    f(w)\leq f_L+\tau(f(w^0)-f_L),    
\end{equation}
where $f_L$ is the best value returned by all the solvers on the 5 multi-start runs, $f(w^0)$ is the starting point of the solvers (it is the same for all the solvers) and $\tau$ is a tolerance. The smaller $\tau$, the more stringent the requirement.

For each  $s\in {\cal S}$ and $p\in {\cal P}$, we define $c_{s,p}$ as the total time needed by solver $s$ to solve problem $p$ and the ratio
$$
r_{s,p}=\frac{c_{s,p}}{\min_{s\in S}\{c_{s,p}\}}
$$
which represents the performance of solver $s$ on solving problem $p$ compared to the best-performing solver on problem $p$. Thus $r_{s,p}= 1$ if the solver $s$ solves the problem within $\min_{s\in S}\{c_{s,p}\}$ seconds, and $r_{s,p}>1$ otherwise. Then, the \emph{performance profile} $\rho_s(\alpha)$ of a solver $s$ is given by
$$
\rho_s(\alpha)=\frac{1}{|{\cal P}|}\left|\{p\in {\cal P}\;|\;r_{s,p}\leq \alpha\}\right| \qquad \alpha\ge 1, .
$$
Performance profiles seek to capture how well the
solver performs compared to the other solvers in $\cal S$ on the set of problems in $\cal P$.
When $\alpha=1$, $\rho_s(\alpha)$ is the fraction of problems for which solver $s$ performs
the best, when $\alpha$ is sufficiently large,
it is the fraction of problems solved by
$s$. Solvers with high values for $\rho_s(\alpha)$ are preferable \cite{more2009benchmarking}.

In particular, in Figure \ref{fig:Performance_profile} we report the performance profiles of the four algorithms when left running for 100 seconds. Results are reported for decreasing values of $\tau$,  $10^{-1},10^{-2}$ and $10^{-4}$, cumulative on all the networks but separating small-size and big-size datasets, with the former on the left and the latter on the right of the figures.

\begin{figure}
    \centering     %%% not \center
        \subfigure[Small datasets: $\tau=10^{-1}$]
            {\label{fig:a}
            \includegraphics[width=0.48\textwidth]{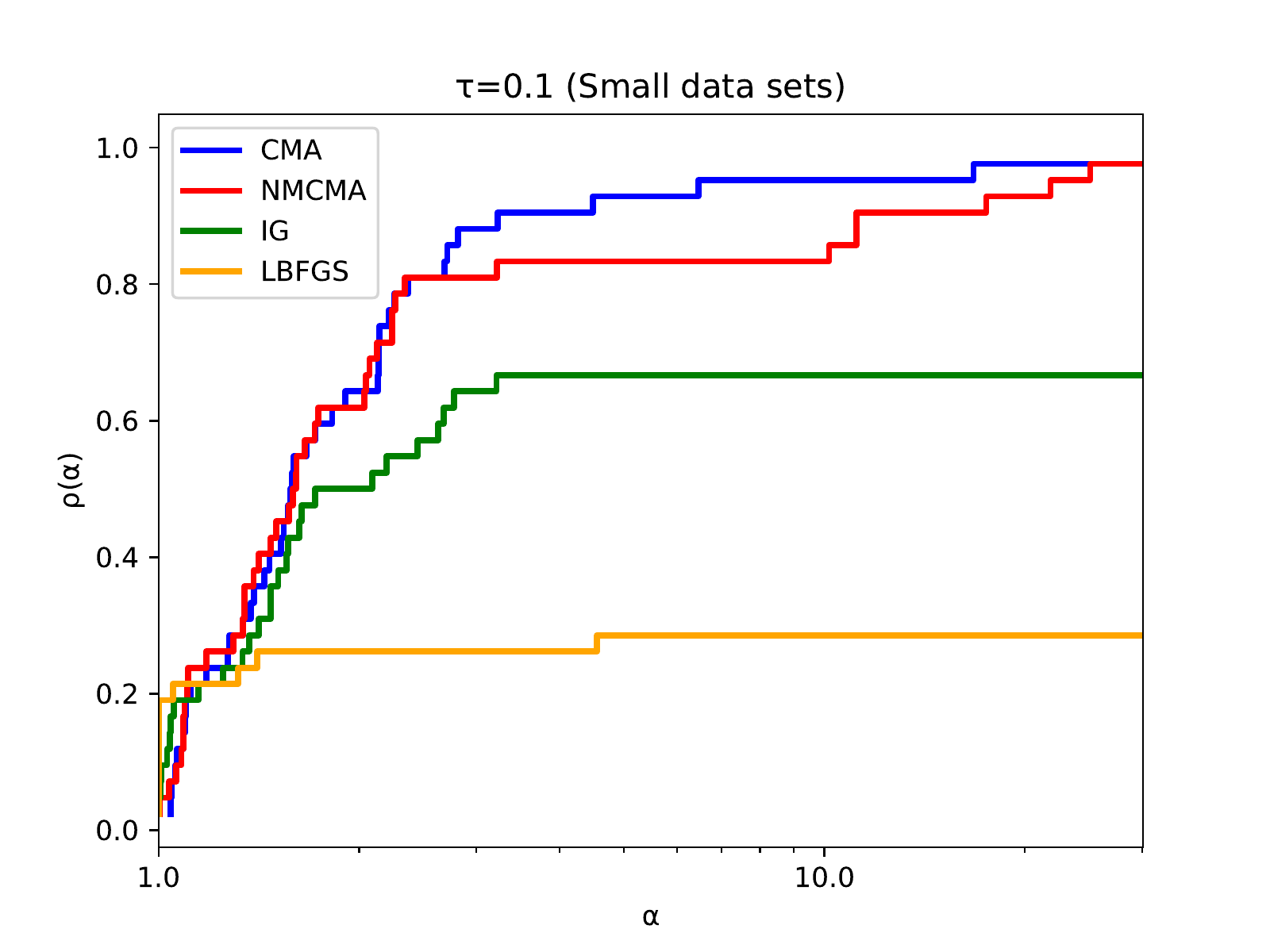}}
        \subfigure[big-size datasets: $\tau=10^{-1}$]
            {\label{fig:b}
            \includegraphics[width=0.48\textwidth]{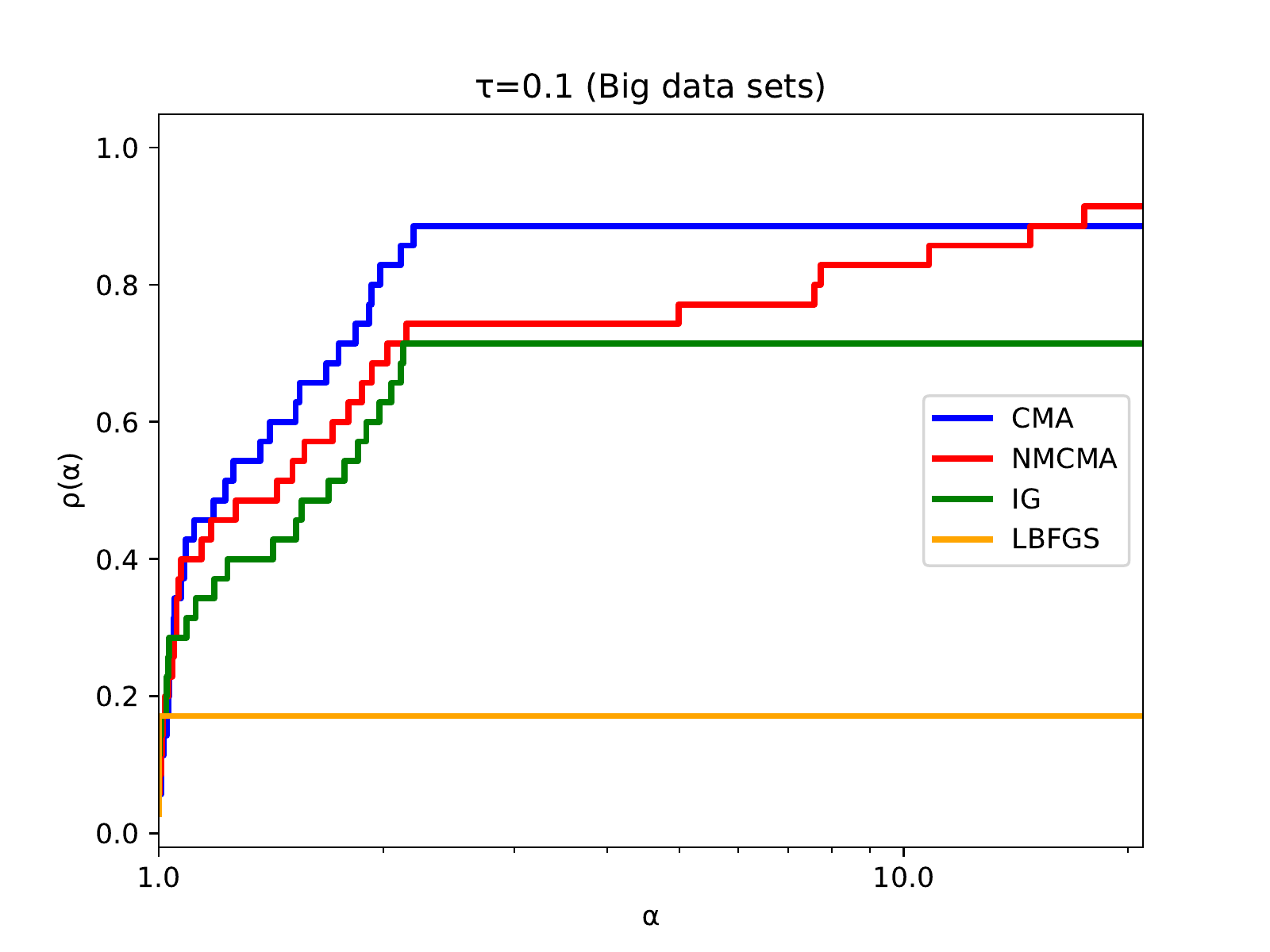}}
            \hfill
            \subfigure[Small datasets $\tau=10^{-2}$]
            {\label{fig:c}
            \includegraphics[width=0.48\textwidth]{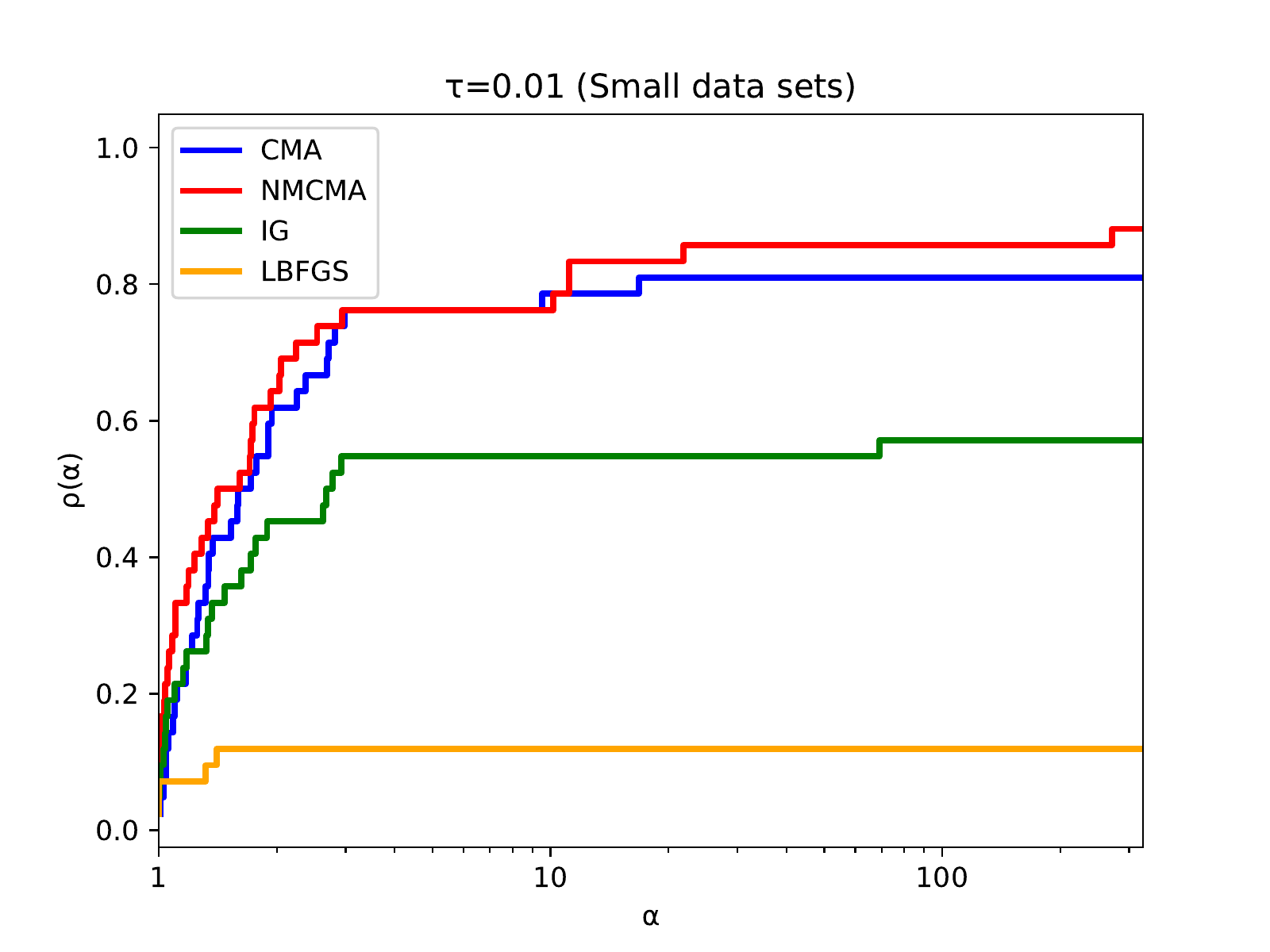}}
        \subfigure[big-size datasets $\tau=10^{-2}$]
            {\label{fig:d}
            \includegraphics[width=0.48\textwidth]{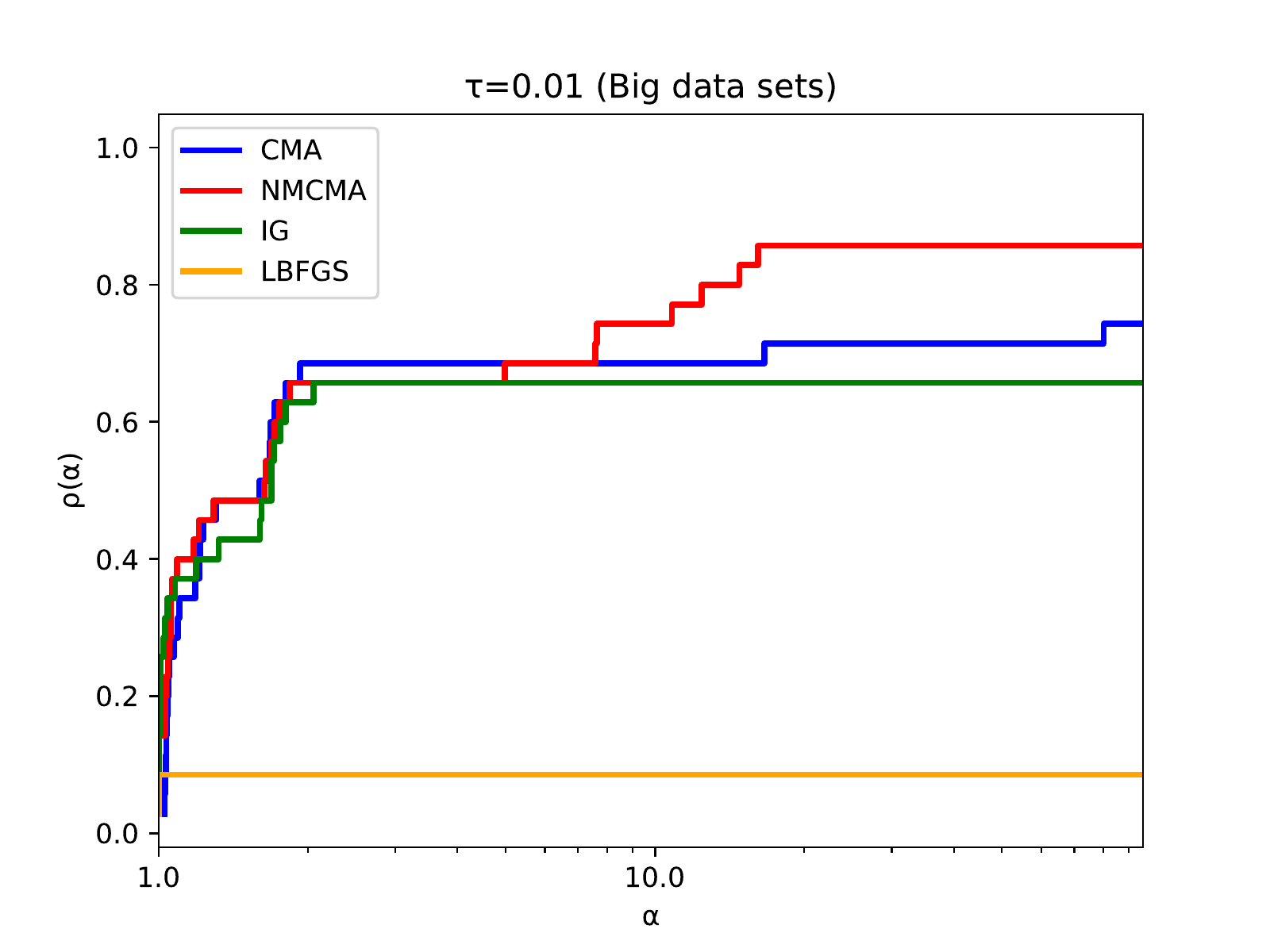}}
            \hfill
            \subfigure[Small datasets $\tau=10^{-4}$]
            {\label{fig:e}
            \includegraphics[width=0.48\textwidth]{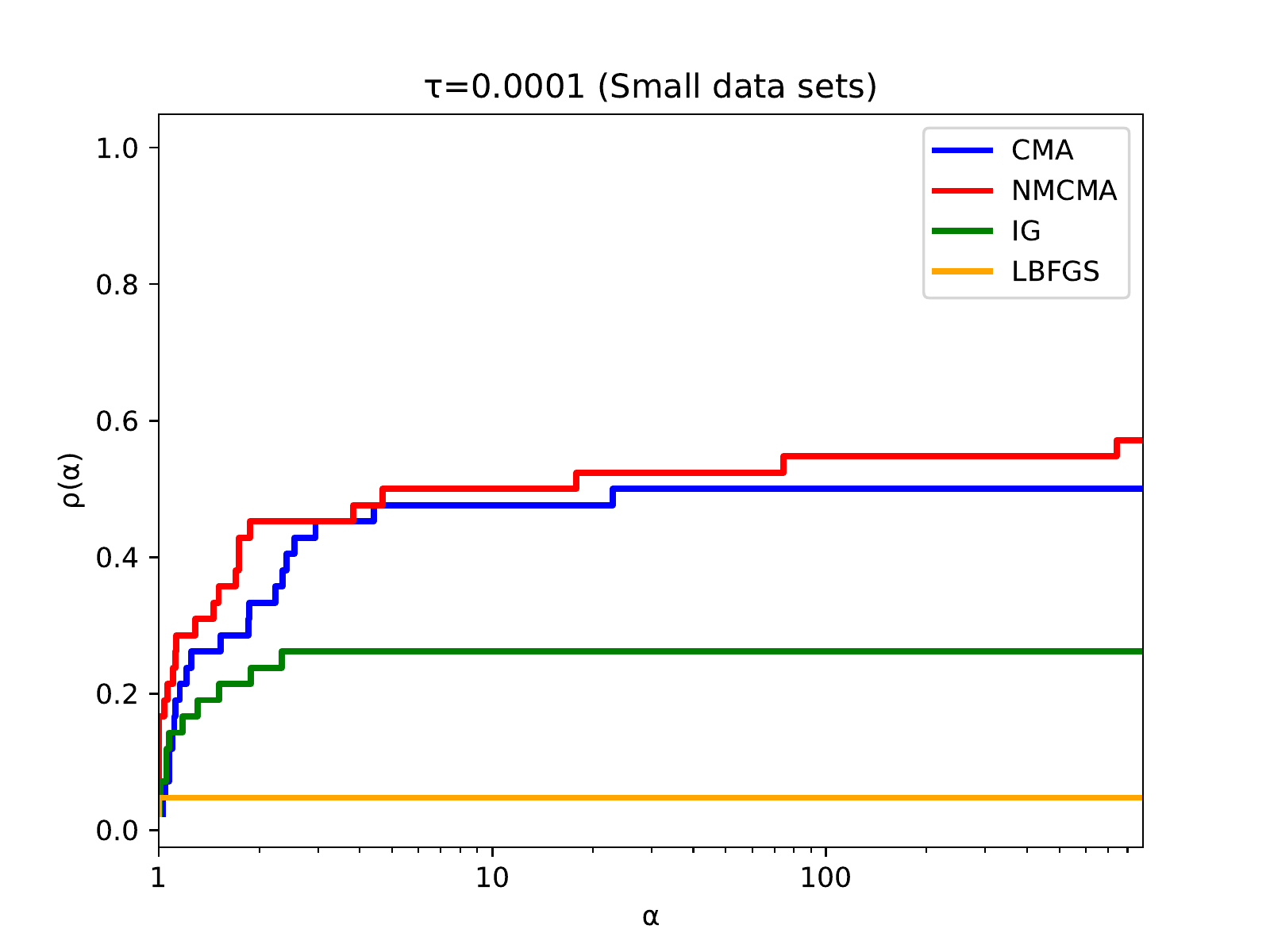}}
        \subfigure[big-size datasets $\tau=10^{-4}$]
            {\label{fig:f}
            \includegraphics[width=0.48\textwidth]{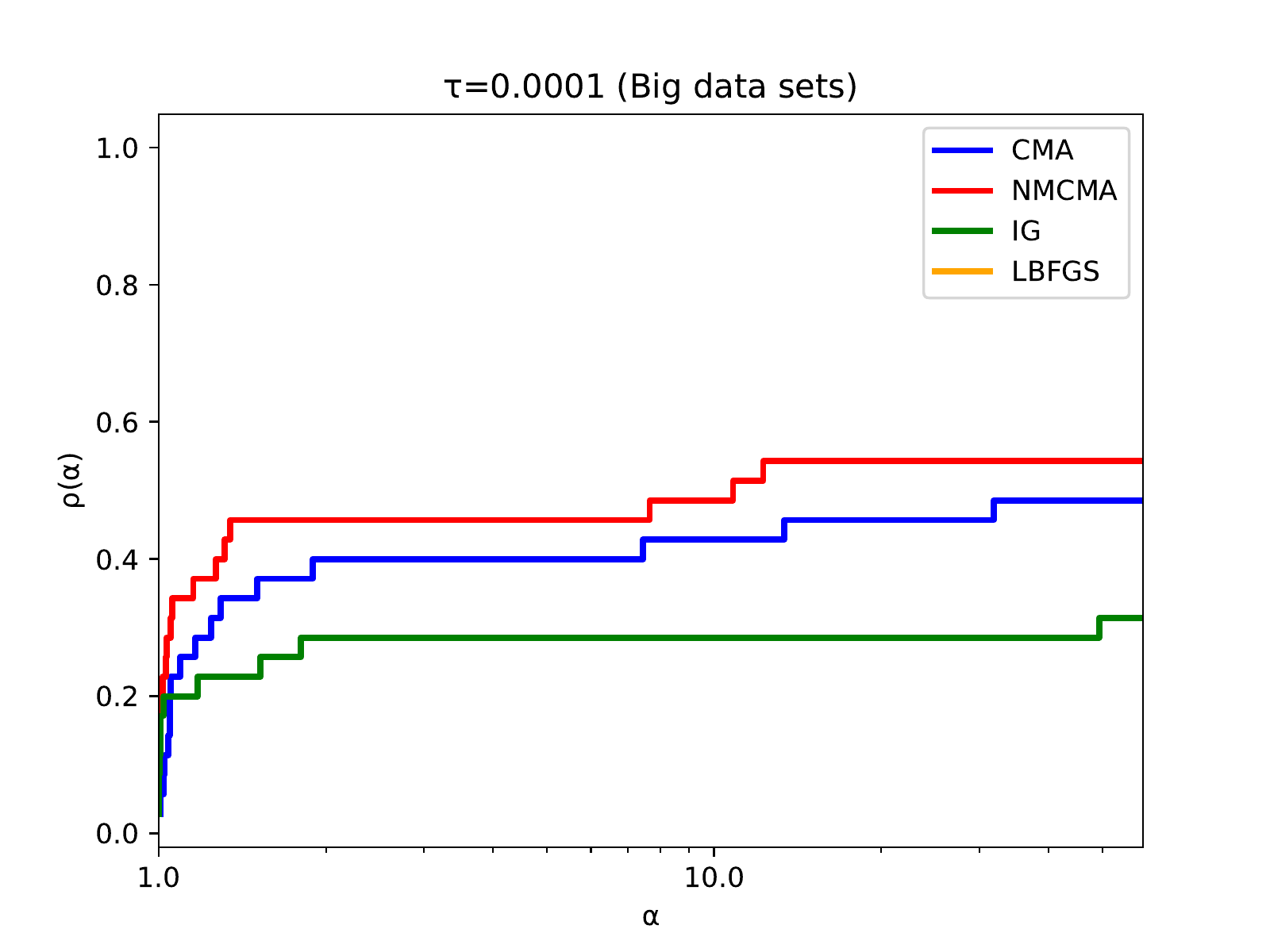}}
    \caption{Performance profiles on all the networks  for different values of $\tau$. Results for small-size datasets are reported to the left, and for big-size datasets to the right.}
    \label{fig:Performance_profile}
\end{figure}

When $\tau=10^{-1}$, \CMA\ and \NMCMA\ have similar performance profiles and they outperformed both IG and L-BFGS methods on small-size and large-size problems. 

{When $\alpha=1$ \CMA\ and \NMCMA\  solve a smaller portion of the problems. For increasing values of $\alpha$ they bothreaches 100 \%, \CMA\  having a slight advantage over its non-monotone counterpart.}

{Conversely, L-BFGS is far away and it solves around one third of the small-size dataset problems and less than 20 \% of the large-size ones, which is an expected result.}
{Indeed, despite its faster theoretical convergence, L-BFGS iteration often requires a large number of function evaluations to become inefficient against mini-batch algorithms.}
For {decreasing} values of $\tau=10^{-2}$ and $\tau=10^{-4}$, {this behaviour is even more evident.}
{When high-precision on large-size dataset is required (see \ref{fig:f}), L-BFGS never happens to be the best solver.}
{It is interesting to notice that \CMA\ and \NMCMA\, despite requiring at least one evaluation of the entire objective function (i.e., on the whole dataset) per iteration, always outperform IG, which never requires to compute the full $f(\omega^k)$.}
{This suggests that checking the quality of the trial direction can in practice speed up the convergence and leading faster to a stationary point.}
{When dealing with larger problems, the advantage of \CMA\ and \NMCMA\ over IG is slightly reduced by the computational burden of evaluating $f(\omega^k)$, but even in that case they both seem to scale up keeping their advantage.}

\paragraph{Reproducibility.} The code implemented for the numerical results presented in this paper can be found at

\noindent
\url{https://github.com/corradocoppola97/CMA_v2}

\section{Conclusions}\label{sec:CMA_conclusion}
In conclusion, in this paper, we have introduced a novel class of optimization algorithms that leverage mini-batch iterations and batch {derivative-free} linesearch procedures to define {two} convergent frameworks with very mild assumptions but {still maintaining} good numerical performance. The assumptions  needed to prove convergence are standard ones in optimization, allowing us to prove convergence of Random Reshuffling and Incremental Gradient methods  under very mild assumptions with a minimal extra cost to pay represented by a few objective function evaluations. The algorithm encompasses an automatic {control} on the stepsize in each epoch {which provides a controlled updating strategy of the stepsize, leading to no } unneeded fast reductions.
Numerical results suggest that the proposed methods do not deteriorate the performance and they are instead effective in improving the solutions obtained, leading to {good} values of the objective function {since the first iterations}.

\section{Appendix}

\textbf{Proof of Lemma \ref{prop:Lipschitz}}
\begin{proof} 
For the sake of simplicity {we consider the permutation $I^k=\{1\dots P\}$.}
From (\ref{eq:basicIG}), we have that, $\widetilde w_1 = \widetilde w_0 - \zeta \nabla f_{1}(\widetilde w_0)$.
Hence, %using condition (\ref{ass:bertsekas}), 
we obtain
\begin{equation}\label{boundnorm_10_0}
\|\widetilde w_1 - \widetilde w_0\| = \zeta\|\nabla f_{1}(\widetilde w_0)\| %\leq \zeta^k(C + D\|\nabla f(w_k)\|).
\end{equation}
and  \eqref{boundnorm_i_noK_2} holds for $i=1$.

Now, using \eqref{eq:basicIG2} 
and the Lipschitz assumption we can write for all $i$
$$\begin{array}{rl}
     \|\widetilde w_i - \widetilde w_0\|&=\displaystyle\left\| -\zeta \sum_{j=1}^i \nabla f_j(\widetilde w_{j-1})\right\|  \\
    & = \displaystyle\zeta\left\|  \sum_{j=1}^i \left(\nabla f_j(\widetilde w_{j-1})-\nabla f_j(\widetilde w_0)+\nabla f_j(\widetilde w_0)\right)\right\|  \\
    & \le \displaystyle\zeta\sum_{j=1}^i\left\|   \nabla f_j(\widetilde w_{j-1})-\nabla f_j(\widetilde w_0)+\nabla f_j(\widetilde w_0)\right\|  \\
    & \le \displaystyle\zeta\sum_{j=1}^i \left(\left\|  \nabla f_j(\widetilde w_{j-1})-\nabla f_j(\widetilde w_0)\right\|+\left\|\nabla f_j(\widetilde w_0)\right\| \right) \\ 
    & \le \displaystyle\zeta \sum_{j=1}^i\left(L \left\|  \widetilde w_{j-1}-\widetilde w_0\right\|+\left\|\nabla f_j(\widetilde w_0)\right\| \right) \\
\end{array}
$$

which concludes the proof.
\end{proof}

\begin{lemma}[Lemma 1 from \cite{grippo2007}]\label{lemma1[19]}
Assume that a constant $L^\prime>0$ exists such that 
$$\left|f(u)-f(v)\right|\le L^\prime \|u-v\|. $$
Let $\{w^k\}$ be a sequence such that
\begin{equation}\label{lemma1:eq1}
   f(w^{k+1}) \leq R^k - \sigma(\|w^{k+1}-w^k\|)
\end{equation}
where $\sigma:\Re^+\to\Re^+$ is a forcing function and $R^k$ is a reference value that satisfies
\begin{equation}\label{lemma1:eq2}
f(w^k) \leq R^k \leq \max_{0\leq j \leq \min\{k,M\}}\{f(w^{k-j})\},
\end{equation}
for a given integer $M\ge 0$. Suppose that $f$ is bounded below, and that it is Lipschitz continuous on ${\cal L}^0$, that is, there exists a constant $L$ such that
\begin{equation}\label{lemma1:eq3}
|f(x)-f(y)| \leq L \|x-y\|,\quad\mbox{for all},\ x,y\in{\cal L}^0.
\end{equation}%\todo{attenzione ipotesi ripetuta due volte}
Furthermore, for each $k\ge 0$, let $\ell(k)$ be an integer such that $k-\min\{k,M\}\le\ell(k)\le k$ and that
\[
  f(w^{\ell(k)}) = \max_{0\leq j \leq \min\{k,M\}}\{f(w^{k-j})\}.
\]
Then, we have:
\begin{itemize}
    \item[i)] $w^k\in {\cal L}^0$ for all $k$;
    \item[ii)] $f(w^{k+1}) \leq f(w^{\ell(k)}) - \sigma(\|w^{k+1}-w^k\|)$;
    \item[iii)] $\lim_{k\to\infty}\|w^k-w^{\widehat\ell(k)}\| = 0$, where $\widehat\ell(k) = \ell(k+M+1)$;
    \item[iv)] the sequence $\{f(w^k)\}$ is convergent;
    \item[v)] $\lim_{k\to\infty}\|w^{k+1}-w^k\|=0$.
\end{itemize}
\end{lemma}

\begin{proof}
For each $k \ge 0$, let $\ell(k)$ be an integer such that $k-\min(k,M) \le \ell(k) \le k$ and
that
$$f (w^{\ell(k)}) = \max_{0\le j\le \min(k,M)}
[f (w^{k-j} )].$$
Then, using \ref{lemma1:eq1} we get  
\begin{equation}\label{eq:lemma1_1}
f (w^{k+1}) \le  f (w^{\ell(k)})-\sigma \left (\left\|w^{k+1}-w^k\right\|\right).    
\end{equation}
which proves assertion ii).
As a first step we show that the sequence $\{f(w^{\ell(k)})\}$ is non increasing. Indeed, consider  the index $k+1-\min(k+1,M) \le \ell(k+1) \le k+1$ and using 
$\min(k+1,M)\le\min(k,M)+1 $, we get
 $$\begin{array}{l}
 \displaystyle
 f (w^{\ell(k+1)})= \max_{0\le j\le \min(k+1,M)}
[f (w^{k+1-j} )]\le  \max_{0\le j\le \min(k,M)+1}
[f (w^{k+1-j} )] =\\\displaystyle
=\max\left\{f (w^{\ell(k)}), f (w^{k+1})\right\}=f (w^{\ell(k)})\end{array}
 $$
where the last equality follows from \ref{eq:lemma1_1}.

As $\{f(w^{\ell(k)})\}$ is non increasing, and $w^{\ell(0)}=w^0$, we get also that $f(w^k)\le f(w^0)$, so that $w^k\in {\cal L}^0$ for all $k$ which proves assertion i). Further the non increasing $\{f(w^{\ell(k)})\}$ sequence admits a limit for $k\to \infty$. 

Now we observe that if iii) holds then we easily get iv) and v). Indeed by the Lipschitz assumption written on points $w^k$ and $w^{\ell(k)}$ and the fact that $\{f(w^{\ell(k)})\}$ sequence admits a limit we get
$$\lim_{k\to \infty} f(w^k)= \lim_{k\to \infty}  f(w^{\widehat\ell(k)})= \lim_{k\to \infty}  f(w^{\ell(k+M+1)})= \lim_{k\to \infty}  f(w^{\ell(k)}) $$
which proves iv). Then by assumptions \eqref{lemma1:eq1} and \eqref{lemma1:eq2} we get also v).

Thus, it remains to prove iii). 
We can write

$$w^{\widehat\ell(k)} =  \left(w^{\widehat\ell(k)}-w^{\widehat\ell(k)-1}\right) +
\left(w^{\widehat\ell(k)-1}-w^{\widehat\ell(k)-2}\right) +\dots+
\left(w^{k+1}-w^k\right)+w^k$$
so that

\begin{equation}
    \label{grippo2007:eq0}
\lim_{k\to \infty}\|w^{\widehat\ell(k)}-w^k\|=\lim_{k\to \infty}\left\|\sum_{k=1}^{\widehat\ell(k)-k} \left(
w^{\widehat\ell(k)-j+1}-w^{\widehat\ell(k)-j}
\right)
\right\|    
\end{equation}

where $\widehat\ell(k)-k\le M+1$.
We  prove by induction on $1\le j\le M+1$ that 
\begin{equation}\label{grippo2007:eq1}
\lim_{k\to \infty} \left\|
w^{\ell(k)-j+1}-w^{\ell(k)-j}
\right\| =0.    
\end{equation}
Let consider the ii) written for $k=\ell(k)-j-1$ we have
\begin{equation}
\label{grippo2007:eq3}
f(w^{\ell(k)-j+1})\le f(w^{\ell(\ell(k)-j)})-\sigma \left(\left\|w^{\ell(k)-j+1}-w^{\ell(k)-j}  \right\|\right) 
\end{equation}

Let $j=1$, taking the limit, using   $f(w^{\ell(k)})-f(w^{\ell(\ell(k)-1)})\to 0$ and the definition of forcing function, we get
$$\lim_{k\to \infty} \left\|
w^{\ell(k)}-w^{\ell(k)-1}
\right\| =0$$
which is iii) with $j=1$. Using the Lipschitz continuity, we also get for $j=1$ that
\begin{equation}
    \label{grippo2007:eq2}
\lim_{k\to \infty} f(w^{\ell(k)-j})=\lim_{k\to \infty} f(w^{\ell(k)})    
\end{equation}
Now we assume that \eqref{grippo2007:eq1} \eqref{grippo2007:eq2} hold for $j$ and we prove that they are valid for $j+1$.

By  \eqref{grippo2007:eq3} with $j+1$,  using \eqref{grippo2007:eq2} and the definition of forcing function we get
$$\lim_{k\to \infty} \left\|
w^{\ell(k)-j}-w^{\ell(k)-j-1}
\right\| =0$$
Now using the Lipschitz continuity and
\eqref{grippo2007:eq2} we get that \eqref{grippo2007:eq2} holds for $j+1$ too. This terminates the induction.

Now from \eqref{grippo2007:eq0} we get 
$$
\lim_{k\to \infty}\|w^{\widehat\ell(k)}-w^k\|=0$$
\end{proof}
\section*{Acknowledgement} Authors wish to thank Peter 
 Richt{\'a}rik for pointing out some relevant references.

\section*{Data availability statement}

The data that support the findings of this paper are available directly by reproducing them on the public Github repository \url{https://github.com/corradocoppola97/CMA_v2}. 
Furthermore, for each run performed on our hardware, a \texttt{.txt} file with the history log is available and can be provided by the corresponding author.

\section*{Conflicts of interest statement}

The authors certify that they have NO affiliations with or involvement in any
organization or entity with any financial interest (such as honoraria; educational grants; participation in speakers’ bureaus;
membership, employment, consultancies, stock ownership, or other equity interest; and expert testimony or patent-licensing
arrangements), or non-financial interest (such as personal or professional relationships, affiliations, knowledge or beliefs) in
the subject matter or materials discussed in this manuscript.

\newpage
\bibliography{sn-bibliography}% common bib file

\end{document}